\documentclass[11pt]{article}

\usepackage[a4paper,margin=1in]{geometry}
\usepackage{amsmath,amssymb,amsthm,mathtools}
\usepackage{enumitem}
\usepackage{microtype}
\usepackage[hidelinks]{hyperref}

\newcommand{\C}{\mathbb C}
\newcommand{\R}{\mathcal R}
\newcommand{\K}{\mathcal K}
\newcommand{\im}{\operatorname{Im}}
\newcommand{\ord}{\operatorname{ord}}
\newcommand{\Mat}{\operatorname{M}}

\theoremstyle{definition}
\newtheorem{definition}{Definition}[section]

\newtheorem{lemma}[definition]{Lemma}
\theoremstyle{plain}
\newtheorem{theorem}[definition]{Theorem}
 
\theoremstyle{remark}

\title{ On computing  deformations of contraction algebras of type $E$}
\author{ Agata Smoktunowicz,  Alicja Smoktunowicz}
\date{}

\begin{document}
\maketitle

\begin{abstract}
We construct deformations of contraction algebras of type $E$. We consider the
class given by the potential
\[
  x^3+xy^3+p(y),
\]
where $0\neq p(y)\in\C[y]$ and $\ord(p)=4$. We show that the contraction
algebras of this type deform to
\[
  \C^{\oplus 6}\oplus \Mat_2(\C)^{\oplus 3}\oplus \Mat_3(\C).
\]
We use a slight generalisation of the second method described in Section~11 of
\cite{SmoktunowiczDeformations}. In particular, the construction does not require
the generator images to lie in $M[t,t^{-1}]$, and the auxiliary polynomials need
not be square-free.
 We then consider the class of contraction algebras given by the potential 
 \[
  x^3+xy^3+p(y),
\]
where $0\neq p(y)\in\C[y]$ and $\ord(p)=n>4$. We show that the contraction
algebras of this type deform to
\[
  \C^{\oplus 7}\oplus \Mat_2(\C)^{\oplus n-1}\oplus \Mat_3(\C).
\]

\end{abstract}
\section{Introduction}
 This paper is motivated by questions asked by Michael Wemyss about deformations of  contraction algebras. These questions ask whether every contraction algebra can be flatly deformed to a unique semisimple algebra suggested by geometry.  
One aim is to give an intrinsic way of introducing algebraic-geometric concepts using only noncommutative algebra.   A concrete aim would be to read Gopakumar-Vafa invariants directly  from a given contraction algebra using  algebraic methods.    A related result in the literature concerns a generic geometric deformation of the
contraction algebras (Theorem $1.2$, \cite{HuaToda}).  In this situation, the corresponding contraction algebra becomes semisimple, and its 
matrix-block multiplicities recover the genus-zero Gopakumar-Vafa invariants. However, in general it is not known if geometric and algebraic deformations coincide.  Toda constructed geometric deformations for all contraction algebras, but it is not clear whether these families are flat \cite{Toda}.  

A related question about uniqueness of flat deformations of contraction algebras  to semisimple algebras was also asked by Ben Davison. 

These questions were answered in the affirmative for contraction algebras of types $A$ and $D$ \cite{JelisiejewSmoktunowicz}, however the case of contraction algebras of type $E$ remained open. There are four types of contraction algebras of type $E$ (see \cite{BW}), and the aim of our paper is to show that the first class of contraction algebras of type $E$ admits flat deformations to algebras suggested by geometry. 
  
 Contraction algebras are finite-dimensional noncommutative algebras naturally associated
with three-dimensional flopping contractions. They were introduced by Donovan and Wemyss
in \cite{DW} as universal noncommutative deformation algebras of the contracted curves, and they encode
significantly more information than their commutative quotients. In particular, contraction
algebras play an important role in the homological and birational geometry of threefold flops;
they control the corresponding noncommutative twist functors and are closely related to the
classification of flopping contractions \cite{Ben, IW, MV}. Recent work on noncommutative normal forms also
shows that contraction algebras provide a natural algebraic framework for the ADE
classification of threefold flops, including the more complicated type \(E\) cases \cite{BW}.

An additional motivation comes from enumerative geometry. For a smooth threefold flop
contracting a single curve, Toda related the dimension of the contraction algebra to the
genus-zero Gopakumar--Vafa invariants \(n_j\) by the formula
\[
    \dim_{\mathbb C} A_{\mathrm{con}}
      = \sum_{j\geq 1} j^2 n_j,
\]
and showed that the dimension of its abelianisation is \(n_1\) (Theorem $1.1$, \cite{Toda}). Thus a semisimple algebra of
the form
\[
    \bigoplus_{j\geq 1} M_j(\mathbb C)^{\oplus n_j}
\]
has precisely the dimension predicted by Toda's formula. This suggests an algebraic
problem: instead of recovering only the numerical invariants from the dimension of the
contraction algebra, one may try to construct a flat deformation whose generic fibre is a
semisimple algebra in which the multiplicity of the block \(M_j(\mathbb C)\) is the
corresponding GV multiplicity.
This problem was suggested to us by Michael Wemyss. 

It is therefore useful to have purely algebraic constructions of flat
deformations, especially constructions that can be verified explicitly from generators and
relations. This was one of the motivations for the methods developed in earlier work on
deformations of noncommutative algebras.

In the present paper we study contraction algebras of type \(E\) arising from potentials
\[
    x^3+xy^3+p(y),
\]
 where $\ord (p)\geq 4$ is the smallest power of $y$ which appears as a nonzero summand in $p(y)$.  
 In contrast to the  cases of contraction algebras of  types $A$ and $D$, in the expected fibres, besides $1$-by-$1$  and $2$-by-$2$ matrix factors, $3$-by-$3$ matrix rings  occur.

Our first main result concerns the case
\[
    \ord (p) = 4.
\]
We construct explicitly a flat deformation of the corresponding contraction algebra to
\[
    \mathbb C^{\oplus 6}
      \oplus M_2(\mathbb C)^{\oplus 3}
      \oplus M_3(\mathbb C).
\]
The dimensions of these blocks give
\[
    6+3\cdot 2^2+3^2=27,
\]
which agrees with the dimension of the contraction algebra. More importantly, the decomposition
retains the individual matrix sizes and multiplicities rather than only their weighted sum.

We then treat the more general family in which
\[
    \operatorname{ord} (p) = n.
\]
For the classes considered here we construct flat deformations to
\[
    \mathbb C^{\oplus 7}
      \oplus M_2(\mathbb C)^{\oplus (n-1)}
      \oplus M_3(\mathbb C).
\]
The construction is explicit. We build suitable images \(X,Y\) of the generators in a product
of matrix algebras over a Laurent-series field and verify directly the deformed Jacobi
relations. The separate \(1\times 1\), \(2\times 2\), and \(3\times 3\) components are then
combined by a subdirect-product argument.

The deformation method is a slight generalisation of the second method developed previously in \cite{SmoktunowiczDeformations} (where the construction from \cite{DoraDave} was modified).
We introduce two extensions. First, the images of the generators are allowed to have entries
in the Laurent-series algebra \(M\{t\}[t^{-1}]\), rather than being Laurent polynomials in
\(t\). Second, the auxiliary polynomials are not required to be square-free. These modifications
are essential for the type \(E\) constructions considered below and may also be useful for other
contraction algebras and quiver algebras arising in threefold geometry. 
 Many similar open questions appear in the area in connection with threefolds. Therefore, the methods developed here might be useful in investigating deformations of quiver algebras or other classes of type $E$ quivers. 
 See, for example, \cite{W} for applications of deformations of quiver algebras in geometry, and calculations of GV invariants for some  types of quivers which appear in algebraic geometry. 

The paper is organised as follows. In Section~2 we formulate the deformation method and remove
the square-free hypothesis. Section~3 recalls the contraction algebras associated with the
potentials above. Section~4 treats the case \(\ord (p)=4\). The remaining sections
treat the general cases according to the congruence class of \(\ord (p)-1\) modulo
\(3\).

\section{The method}\label{M1} 

Put
\[
  \R=\C[[t]],\qquad \K=\R[t^{-1}].
\]
Thus $\K$ is the field of Laurent series associated with the power-series ring
$\R$. By ${\mathbb C}[[t]]$ we denote the power-series ring over ${\mathbb C}$.

Let $k'\geq 1$ and let $n_1,\dots,n_{k'}$ be positive integers. For each
$i\leq k'$, fix a monic polynomial
\[
  g_i(z)\in \R[z],\qquad \deg g_i=d_i>0.
\]
Let
\[
  B_i=\K[z]/(g_i(z)),\qquad u_i=z+(g_i(z))\in B_i.
\]
Then $g_i(u_i)=0$ in $B_i$.

\begin{definition}\label{111}
Define
\[
  A'=\bigoplus_{i=1}^{k'}\Mat_{n_i}(B_i)
\]
and
\[
  M\{t\}
  =\bigoplus_{i=1}^{k'}
  \Mat_{n_i}\bigl(\R+\R u_i+\cdots+\R u_i^{d_i-1}\bigr).
\]
Since every $g_i$ is monic, $M\{t\}$ is an $\R$-algebra. Moreover,
\[
  A'=M\{t\}[t^{-1}].
\]
Set
\[
  N=\sum_{i=1}^{k'}n_i^2d_i
\]
and
\[
  A=\bigoplus_{i=1}^{k'}\Mat_{n_i}(\C)^{\oplus d_i}.
\]
Then $\dim_{\C}A=N$.
\end{definition}
\begin{definition}\label{213}
Let notation be as in Definition \ref{111}.  Denote  \[A_{i}'=\Mat_{n_i}(B_{i}).\]
 Let 
\[
  f:\R\langle x_1,\dots , x_\rho\rangle \longrightarrow A'.
\]
 We define homomorphisms 
$f_{1}, \ldots , f_{k'}$ to be the projections of $f$ onto $A_{i}'$,
\[f_{i}=\pi _{i}\circ f .\]
 Therefore, $f=f_{1}\oplus \cdots \oplus f_{k'}$. 
\end{definition}
\subsection{Hypotheses}

Let
\[
  f:\R\langle x_1,\dots,x_\rho\rangle \longrightarrow A'
\]
be a homomorphism of unital $\R$-algebras. We impose the following conditions.

\begin{enumerate}[label=\textup{(\arabic*)}]
\item $f(1)=1_{A'}$.
\item There exist $m_1,\dots,m_N\in\C\langle x_1,\dots,x_\rho\rangle$
      such that $f(m_1),\dots,f(m_N)$ are linearly independent over $\R$.
\item There exists $\delta\geq 0$ such that
      \[
        \im(f)\subseteq t^{-\delta}M\{t\}.
      \]
\end{enumerate}

 We will use the same definition of a flat deformation as in \cite{JelisiejewSmoktunowicz} and \cite{SmoktunowiczDeformations}. 
 This is the standard definition of a flat deformation used in algebraic geometry. 

The next statement is the version of the method that we shall use.

\begin{theorem}\label{thm:method}
Assume that $f$ satisfies the hypotheses above and that each $g_i$ is square-free
over an algebraic closure of $\K$. Then there exist
$q_1,\dots,q_N\in\R\langle x_1,\dots,x_\rho\rangle$ such that
\[
  \im(f)=\bigoplus_{i=1}^{N} \R f(q_i).
\]
In particular, $\im(f)$ is a free $\R$-module of rank $N$.

Writing
\[
  f(q_k)f(q_m)
  =\sum_{i=1}^{N}\bigl(\zeta_{i,k,m}+t\xi_{i,k,m}(t)\bigr)f(q_i),
\]
with $\zeta_{i,k,m}\in\C$ and $\xi_{i,k,m}(t)\in\R$, the constants
$\zeta_{i,k,m}$ define a finite-dimensional $\C$-algebra $N_0$ with basis
$d_1,\dots,d_{N}$ by
\[
  d_kd_m=\sum_{i=1}^{N}\zeta_{i,k,m}d_i.
\]
Moreover,
\[
  N_0\cong \im(f)/t\im(f),
\]
and $\im(f)$ is a flat deformation from $N_0$ to $A$.
\end{theorem}

\begin{proof}
The proof is the same as the proof of the corresponding result in
\cite{SmoktunowiczDeformations}, with the following modifications (note also that in the current notation we use $N$ instead of $n$ which was used in \cite{SmoktunowiczDeformations}). 
First, one does not need to assume that every generator image belongs to
$M[t,t^{-1}]$; the boundedness condition
$\im(f)\subseteq t^{-\delta}M\{t\}$ is sufficient. Second, wherever the proof in
\cite{SmoktunowiczDeformations} uses polynomial expressions in $t$, one may use
power series in $\R$ (the detailed changes are at the end of this proof). Thus the elements denoted there by $q_i$ may be taken in
$\R\langle x_1,\dots,x_\rho\rangle$ rather than in a polynomial extension (see the detailed list of changes at the end of this proof).

Finally, the ambient algebra is $M\{t\}[t^{-1}]$, not $M[t,t^{-1}]$.
The remainder of the proof is unchanged.

{\em The list of detailed changes.} We use the proof of Theorem $8$ of \cite{SmoktunowiczDeformations}, with only minimal changes. In  \cite{SmoktunowiczDeformations} $n$ was used instead of $N$, so after these changes we put $n:=N$. The  changes are the same as the changes described in Subsection 11.3 of \cite{SmoktunowiczDeformations},  and additionally we make the following changes: 

\begin{itemize}
\item The assumption that $f(x_{i})\in M[t, t^{-1}]$ is removed from Assumption $1$ of Subsection $11.2$ of \cite{SmoktunowiczDeformations}.
\item In Proposition $9$ (1), we write ``There exist elements $q_{1}, \ldots , q_{n}\in {\mathbb C}\langle x_{1}, \ldots , x_{\rho}\rangle \{t\}$'' instead of  ``There exist elements $q_{1}, \ldots , q_{n}\in  {\mathbb C}\langle x_{1}, \ldots , x_{\rho}\rangle [t]$'' (in the current notation $ {\mathbb C}\langle x_{1}, \ldots , x_{\rho}\rangle \{t\}=\R\langle x_{1}, \ldots , x_{\rho }\rangle$, moreover $ {\mathbb C}\langle x_{1}, \ldots , x_{\rho}\rangle [t]$ is the polynomial ring in the variable $t$  over the free unital ${\mathbb C}$-algebra in free generators $x_{1}, \ldots , x_{\rho }$). We also remove a related part of the proof of Proposition $9$ (1). Namely, we remove lines $3$-$21$ from page $10$ of \cite{SmoktunowiczDeformations}.  
\item At the end of the last line of page $20$ we write ``$i_{4}$, for some $i_{4}\in t\cdot \im (f)$'', instead of ``$i_{3}$''.  
\item In line $-2$ on page $20$  (the second line from the end of page $20$) we write ``$M\{t\}[t^{-1}]$'' instead of ``$M[t, t^{-1}]$''. So we write ``$M\{t\}[t^{-1}]$ is an algebra containing the generator images''.  
\end{itemize} 
\end{proof}

\subsection{Removing the square-free hypothesis}

\begin{theorem}\label{thm:nonsquarefree} 
Assume that $f$ satisfies the assumptions of Theorem~\ref{thm:method}  except that the polynomials $g_i$ are not required to be square-free. Let 
\[
  N_0\cong \im(f)/t\im(f).
\] Then there is a flat deformation from the $\mathbb C$-algebra $N_{0}$ to $A$. 
\end{theorem}

\begin{proof} 
Let \[L=\im (f).\]
 
The construction of  $q_{1}, \ldots , q_{N},$ the freeness of $\im (f)$, and the identification of the
special fibre $N_{0}$ with $\im (f)/t\im (f)$  in Theorem \ref{thm:method}  do not use square-freeness; that hypothesis is used only
in the identification of the generic fibre $A$ (this only appears on page $29$, above Corollary $28$ in \cite{SmoktunowiczDeformations}). Note also that in the current notation we use $N$ instead of $n$ which was used in \cite{SmoktunowiczDeformations}. 

 Hence $L$ is free of rank $N$, \[L[t^{-1}]= A',\] and
\[N_{0}=L/tL.\] 
 So $L=\im (f)$ is a deformation of $N_{0}$, as in Theorem \ref{thm:method}. 

Choose a finite Puiseux extension \[{\mathcal K}\subseteq {\mathcal K} '={\mathbb C}((v)), t=v^{m}\] 
over which all $g_i$ split. Put ${\R} '=\C[[v]]$, 
 the power-series ring in variable $v$. Let \[L'=L\otimes _{\mathcal R}{\mathcal R}'.\]
  
 Since $L$ is a formal deformation of $N_{0}$, $L'$ is a formal deformation of $N_{0}$.
 Indeed, in the formulas for the multiplication of elements in the formal deformation of $N_{0}$ (as in 5.1.4 on page $100$ of  \cite{SW})  we can substitute $t:=v^{m}$, to obtain a formal deformation $L'$ of the algebra $N_{0}$ in variable $v$.

 We will now show that $L'$ is a flat deformation to  the algebra  $A'''$ described below. 

  Since ${\mathcal R}'$ is flat over ${\mathcal R}$, and since $L[t^{-1}]\cong A'$, we have 
\[L'[v^{-1}]\cong A'\otimes _{\mathcal K}{\mathcal K}'.\]
 Write  
\[g_{i}(z)=\prod _{j=1}^{s_{i}}(z-a_{i,j})^{c_{i,j}},\]
 with distinct $a_{i,j}\in K'$, the Chinese remainder theorem gives
\[A'\otimes _{\mathcal K}{\mathcal K}'\cong 
\bigoplus _{i=1}^{k'}M_{n_{i}}(\bigoplus _{j=1}^{s_{i}}{\mathcal K}'[{\varepsilon}]/({\varepsilon }^{c_{i,j}})).\]

Define a finite-dimensional $\C$-algebra
\[
  A'''=
  \bigoplus_{i=1}^{k'}
  \Mat_{n_i}\!\left(
    \bigoplus_{j=1}^{s_i}\C[\varepsilon]/(\varepsilon^{c_{i,j}})
  \right).
\]
Then
\[L'[v^{-1}] \cong A'''\otimes {_\mathbb C}{\mathcal K}',\]
so $L'$ is a deformation of $N_{0}\cong L'/vL'$ to $A'''$  (indeed a split deformation over the parameter $v$).
  
For every $c\geq1$, the algebra $\C[\varepsilon]/(\varepsilon^c)$ flatly
deforms to $\C^{\oplus c}$; for example define 
\[{\mathcal A}_{c} ={\mathbb C}[[v]][z]/(\prod _{j=1}^{c}(z-jv)).\]
At $v=0$, 
\[ {\mathcal A}_{c}/v{\mathcal A}_{c}\cong {\mathbb C}[z]/(z^{c}),\]
 while the roots $v, 2v, \ldots , cv$ are distinct, so by the Chinese remainder theorem
 \[ {\mathcal A}_{c}[v^{-1}]\cong {\mathbb C}((v))^{\oplus c} .\]
  Since the defining polynomial is monic of degree $c$, ${\mathcal A}_{c}$ is free of rank $c$ over ${\mathbb C}[[v]]$.
 Consequently, $A'''$ flatly deforms to
\[
  \bigoplus_{i=1}^{k'}\Mat_{n_i}(\C)^{\oplus d_i}=A.
\]
 Therefore, $N_{0}$ deforms flatly  to $A'''$, and $A'''$ deforms flatly to $A$.  
By transitivity of finite-dimensional deformations
\cite[Theorem~2.20]{JelisiejewSmoktunowicz}, $N_{0}$ deforms flatly  to $A$. 
\end{proof}

\subsection{ Injectivity assumption can be simplified}
  We will use the algebraic analogue of
\emph{Goursat's lemma} \cite{G}. Historically, the result originates in group theory.

\begin{lemma}[Subdirect-product lemma]\label{lem:subdirect}
Let $K$ be a field and let $S\subseteq B_1\oplus\cdots\oplus B_r$ be a
$K$-subalgebra whose projection to each $B_i$ is surjective. Assume that for every pair $i\neq j$ the algebras $B_i$ and $B_j$ have no
\emph{common nonzero quotient algebra}.  Explicitly, this means that there do
not exist two-sided ideals $I_i\lhd B_i$ and $I_j\lhd B_j$ such that
\[
 B_i/I_i\cong B_j/I_j\neq0.
\]
Then
\[
  S=B_1\oplus\cdots\oplus B_r.
\]
\end{lemma}

\begin{proof}
For two factors this is the usual Goursat argument: a proper subdirect product
of $B_1\oplus B_2$ determines ideals $I_i\lhd B_i$ possibly zero, such that 
$B_1/I_1\cong B_2/I_2$. The general case follows by induction.
 A detailed argument: 
We first prove the assertion for two factors.  Let
\[
 S\subseteq A\oplus B
\]
project surjectively onto both $A$ and $B$. Define
\[
 I=\{a\in A:(a,0)\in S\},
 \qquad
 J=\{b\in B:(0,b)\in S\}.
\]
Then $I\lhd A$ and $J\lhd B$ are two-sided ideals (they are two-sided since $(c, d)(a,0) = (ca,0)\in S$,
$(a, 0)(c,d) = (ac,0) \in S$).  For example, if $a\in I$
and $c\in A$, choose $d\in B$ such that $(c,d)\in S$. 
%

We now construct the Goursat isomorphism
\[
 \varphi:A/I\longrightarrow B/J.
\]
Given $a\in A$, choose $b\in B$ with $(a,b)\in S$ and put
\[
 \varphi(a+I)=b+J.
\]
This is well defined.  Indeed, if $(a,b),(a,b')\in S$, then
$(0,b-b')\in S$, so $b-b'\in J$.  If $a'-a\in I$ and
$(a,b),(a',b')\in S$, then $(a'-a,0)\in S$ and hence
$(0,b'-b)\in S$, so again $b'-b\in J$.

The map $\varphi$ is a $K$-algebra homomorphism because $S$ is closed under
addition, scalar multiplication and multiplication.  It is surjective because
the projection $S\to B$ is surjective.  It is injective as follows: if
$\varphi(a+I)=0$, choose $(a,b)\in S$.  Then $b\in J$, so $(0,b)\in S$, and
therefore $(a,0)=(a,b)-(0,b)\in S$; hence $a\in I$.  Thus
\begin{equation}\label{eq:goursatquotient}
 A/I\cong B/J.
\end{equation}
More precisely,
\[
 S=\{(a,b)\in A\oplus B:\varphi(a+I)=b+J\},
\]
so $S$ is the fibre product of $A$ and $B$ over their common quotient.

If $S$ is a proper subalgebra of $A\oplus B$, then the quotient in
\eqref{eq:goursatquotient} is nonzero.  Indeed, if $A/I=0$, then $I=A$ and,
by the isomorphism, $J=B$.  Hence $A\oplus0\subseteq S$ and
$0\oplus B\subseteq S$, which gives $S=A\oplus B$, a contradiction.  Notice
that the ideals $I,J$ themselves need not be nonzero; what is forced to be
nonzero for a proper subdirect product is the common quotient $A/I\cong B/J$.
This proves the two-factor case.

We prove the general case by induction on $r$.  Assume the result for $r-1$
factors and put
\[
 C=B_1\oplus\cdots\oplus B_{r-1}.
\]
Let $T$ be the image of $S$ under the projection onto $C$.  The projection of
$T$ onto each $B_i$, $i<r$, is surjective, so the induction hypothesis gives
$T=C$.  Hence $S$ is a subdirect product of the two algebras $C$ and $B_r$.
If $S\neq C\oplus B_r$, the two-factor case produces ideals
$I\lhd C$, $J\lhd B_r$ and a common nonzero quotient
\[
 C/I\cong B_r/J\neq0.
\]
Every two-sided ideal of the finite direct product $C$ has the form
\[
 I=I_1\oplus\cdots\oplus I_{r-1},
 \qquad I_i\lhd B_i,
\]
so
\[
 C/I\cong
 (B_1/I_1)\oplus\cdots\oplus(B_{r-1}/I_{r-1}).
\]
Because $C/I\neq0$, at least one $B_i/I_i$ is nonzero.  Composing the
surjection
\[
 B_r\twoheadrightarrow B_r/J\cong C/I
\]
with the projection $C/I\twoheadrightarrow B_i/I_i$ shows that $B_i/I_i$ is a
common nonzero quotient of $B_i$ and $B_r$, contrary to the hypothesis.  Thus
$S=C\oplus B_r$, completing the induction.
\end{proof}

\begin{theorem}\label{New}
Assume that $f$ satisfies the assumptions of Theorem \ref{thm:nonsquarefree} except for the assumption that $\im (f)$ has $N$ independent images. 
 Let $A_{1}', \ldots , A_{k'}$ and $f_{1}, \ldots , f_{k'}$ be as in Definition \ref{213}. Here,  $A_{i}'=M_{n_{i}}(B_{i})$ with $B_{i}={\K}[z]/(g_{i}(z)).$  Assume that for each $i\leq k'$, $\dim_{\K}\K\im (f_{i})=\dim _{\K} A_{i}'$. Assume that $n_{i}\neq n_{j}$ for $i\neq j$ for $i,j\leq k'$. Denote 
\[
  N_0\cong \im(f)/t\im(f).
\] Then all the assumptions of Theorem ~\ref{thm:nonsquarefree} hold, and there is a flat deformation from the $\mathbb C$-algebra $N_{0}$ to $A$. 
\end{theorem}
\begin{proof}  Recall that $\dim _{\K}A'=\sum_{i=1}^{k'}\dim_{\K}A_{i}=N$.  We apply Lemma \ref{lem:subdirect}
 to $S=\K\im (f)$, $B_{i}=\K\im (f_{i})$ for all $i\leq k'$. The hypothesis of our theorem gives $\K\im (f_{i})=A_{i}'$ for every component, and Lemma \ref{lem:subdirect} gives surjectivity onto the full product (once we verify its hypothesis). Therefore, \[\dim _{\K}\K\im (f)=\sum_{i=1}^{k'}\dim _{\K}\K\im (f_{i})=\dim _{\K} A'=N.\]
  Therefore, one can select $N$ words $m_{1}, \ldots , m_{N}$ whose images $f(m_{1}), \ldots , f(m_{N})$ form a $\K$-basis of $\im (f)$, 
 so the assumptions of Theorem \ref{thm:nonsquarefree} hold.
We will now  check that the assumptions of Lemma  \ref{lem:subdirect} hold, so that $\K\im (f_{1}),\ldots ,  \K\im (f_{k'})$ have no common nonzero quotients, so it can be applied. 
  Every two-sided ideal of $M_{n_i}(B_i)$ is $M_{n_i}(I)$ for an ideal $I\lhd B_i$, so every quotient has the form
\[
M_{n_i}(B_{i}/I).
\]
 Therefore, every nonzero quotient of $M_{n_i}(B_{i})$ satisfies the standard polynomial identity of $n_i$-by-$n_i$ matrices, and it does not satisfy the standard polynomial identity of $n_{j}$-by-$n_{j}$ matrices for $n_{j}<n_{i}$ (since $B_{i}$ is a unital algebra). Therefore, for $i> j$, $M_{n_i}(B_{i})$ and $M_{n_j}(B_{j})$ do not have a  common nonzero quotient.  This proves that Lemma \ref{lem:subdirect} can be applied.
\end{proof}

\section{Deformations from the contraction algebra to the semisimple algebra}

Let $\C\langle x,y\rangle$ denote the free unital associative $\C$-algebra on
$x,y$. Let
\[
  p(y)\in\C[y],\qquad 0\neq p(y),\qquad \ord(p)=4,
\]
and put
\[
  g(y)=p'(y).
\]
Then
\[
  g(y)=a_0y^3+y^4\widetilde g(y),
  \qquad a_0\in\C^\times,
\]
for some $\widetilde g(y)\in\C[y]$.

Consider the potential
\[
  x^3+xy^3+p(y).
\]
Its Jacobi relations are
\[
  3x^2+y^3=0
\]
and
\[
  y^2x+yxy+xy^2+g(y)=0.
\]
Let $I$ be the ideal generated by these two elements and let
\[
  \mathfrak m=(x,y)\subset\C\langle x,y\rangle.
\]
The completed contraction algebra is
\[
  N=\C\langle x,y\rangle/I',
  \qquad
  I'=\bigcap_{r\geq1}(I+\mathfrak m^r).
\]
We use the known facts \cite{MW} that $N$ is local, that the ideal generated by the images
of $x$ and $y$ is nilpotent, and that
\[
  \dim_{\C}N=27=6+3\cdot4+9.
\]

\begin{lemma}\label{nieuzyte}
One has
\[
  I'=I+\mathfrak m^{29}.
\]
Equivalently,
\[
 N\cong\C\langle x,y\rangle/(I+\mathfrak m^{29}).
\]
\end{lemma}

\begin{proof}
 Let $J$ be the maximal ideal of $N$. Since $N$ is local and finite-dimensional,
$J^r=J^{r+1}\neq0$ cannot occur: otherwise Nakayama's lemma applied to the
$N$-module $J^r$ would give $J^r=0$. Hence the nonzero powers of $J$ form a
strictly decreasing chain. Since $\dim_{\C}N=27$, certainly $J^{29}=0$.
Therefore $\mathfrak m^{29}\subseteq I'$, so
$I+\mathfrak m^{29}\subseteq I'$. The reverse inclusion follows immediately
from the definition of $I'$ by taking $r=29$.
\end{proof}

\section{ Contraction algebras with $\ord(p)=4$}\label{Section3}
Our target algebra is
\[
  A=\C^{\oplus6}\oplus\Mat_2(\C)^{\oplus3}\oplus\Mat_3(\C).
\]
We now construct a homomorphism to which Theorem~\ref{New}
applies.

\subsection{Preliminary notation and the common deformed relations}

Write $g$ according to congruence classes of exponents modulo $3$:
\begin{equation}\label{eq:gdecomp1}
  g(y)=g_0(y^3)+g_1(y^3)y+g_2(y^3)y^2,
\end{equation}
where $g_0,g_1,g_2\in\C[z]$. Since $g(y)\in y^3\C[y]$, one has
\[
  g_0(0)=g_1(0)=g_2(0)=0,
\]
and, because the coefficient of $y^3$ in $g$ is $a_0\neq0$,
\[
  g_0(z)=a_0z+O(z^2).
\]

Choose once and for all
\[
  \sigma\in\C\setminus\{0,1\}
\]
and a primitive cube root of unity $\omega\in\C$, so
\[
  \omega^3=1,\qquad \omega\neq1,
  \qquad 1+\omega+\omega^2=0.
\]
Define
\begin{align}
  s(t)&=-g_0(\sigma t^3),\\
  r(t)&=-g_1(\sigma t^3),\\
  \ell(t)&=-g_2(\sigma t^3),\\
  \gamma(t)&=2\bigl(g_2(t^3)-g_2(\sigma t^3)\bigr),\label{eq:gamma}\\
  \alpha(t)&=t^{-3}\bigl(g_0(\sigma t^3)-g_0(t^3)\bigr),\label{eq:alpha}\\
  \beta(t)&=t^{-3}\bigl(g_1(\sigma t^3)-g_1(t^3)\bigr).\label{eq:beta}
\end{align}
Then
\[
  s,r,\ell,\gamma\in t^3\R,
  \qquad \alpha,\beta\in\R.
\]
Moreover,
\[
  \alpha(t)=a_0(\sigma-1)+O(t^3),
\]
so $\alpha(t)$ is a unit of $\R$.

We shall arrange that $f$  annihilates
\begin{equation}\label{eq:Phi1}
  \Phi_1
  =x^2-\gamma(t)x+\frac{1}{3}(y^3-t^3)
\end{equation}
and
\begin{equation}\label{eq:Phi2}
\begin{split}
  \Phi_2={}&y^2x+yxy+xy^2+g(y)
  +s(t)+r(t)y+\ell(t)y^2
  -\frac{3}{2}\gamma(t)y^2.
\end{split}
\end{equation}
Modulo $t$, these relations are equivalent to the two Jacobi relations.
 \subsection{The $3\times3$ component}\label{3by3}
 Let $g_{3}(z) = z.$ 
Then
\[d_{3} = 1, B_{3}=K[z]/(z)  \cong  K, u_{3} = 0.\]
\[A_{3}'= M_{3}(B_{3}) \cong M_{3}(K),\]
so the third component fits the notation of Section \ref{M1} exactly. Therefore, we do not need to use the element $u_{3}$  in our construction and we can put
\[
  A_3'=\Mat_3(\K).
\] 
Define
\[
  X_3=
  \begin{pmatrix}
    \gamma(t)&\alpha(t)&\beta(t)\\
    0&0&0\\
    0&0&0
  \end{pmatrix},
  \qquad
  Y_3=
  \begin{pmatrix}
    0&1&0\\
    0&0&1\\
    t^3&0&0
  \end{pmatrix}.
\]
Let
\[
  f_3:\R\langle x,y\rangle\longrightarrow A_3'
\]
be the unital homomorphism defined by $f_3(x)=X_3$ and $f_3(y)=Y_3$.

We have
\[
  Y_3^3=t^3I_3,
  \qquad
  X_3^2=\gamma(t)X_3.
\]
Hence $f_3(\Phi_1)=0$.

A direct multiplication gives
\begin{equation}\label{eq:f3mixed}
  Y_3^2X_3+X_3Y_3^2+Y_3X_3Y_3
  =t^3\alpha(t)I_3+t^3\beta(t)Y_3+\gamma(t)Y_3^2.
\end{equation}
Since $Y_3^3=t^3I_3$, equation~\eqref{eq:gdecomp1} gives
\[
  g(Y_3)=g_0(t^3)I_3+g_1(t^3)Y_3+g_2(t^3)Y_3^2.
\]
Using the definitions of $s,r,\ell,\alpha,\beta,\gamma$, the coefficients of
$I_3,Y_3,Y_3^2$ in $f_3(\Phi_2)$ all vanish. Thus
\[
  f_3(\Phi_2)=0.
\]

For later reference, some useful products are
\begin{align*}
X_3Y_3&=
\begin{pmatrix}
 t^3\beta&\gamma&\alpha\\0&0&0\\0&0&0
\end{pmatrix},
&
Y_3X_3&=
\begin{pmatrix}
0&0&0\\0&0&0\\t^3\gamma&t^3\alpha&t^3\beta
\end{pmatrix},\\[2mm]
X_3Y_3^2&=
\begin{pmatrix}
 t^3\alpha&t^3\beta&\gamma\\0&0&0\\0&0&0
\end{pmatrix},
&
Y_3X_3Y_3&=
\begin{pmatrix}
0&0&0\\0&0&0\\t^6\beta&t^3\gamma&t^3\alpha
\end{pmatrix},\\[2mm]
Y_3^2X_3&=
\begin{pmatrix}
0&0&0\\t^3\gamma&t^3\alpha&t^3\beta\\0&0&0
\end{pmatrix}.
\end{align*}
Also
\begin{align*}
  X_3Y_3X_3&=t^3\beta(t)X_3,\\
  X_3Y_3^2X_3&=t^3\alpha(t)X_3,\\
  X_3^3&=\gamma(t)^2X_3.
\end{align*}
Two of the longer products that are useful for checking the original matrix
list are
\[
Y_3^2X_3Y_3X_3=
\begin{pmatrix}
0&0&0\\
 t^6\beta\gamma&t^6\alpha\beta&t^6\beta^2\\
0&0&0
\end{pmatrix}
\]
and
\[
Y_3^2X_3Y_3^2=
\begin{pmatrix}
0&0&0\\
 t^6\alpha&t^6\beta&t^3\gamma\\
0&0&0
\end{pmatrix}.
\]

\begin{lemma}\label{lem:f3rank}
The $\K$-algebra generated by $X_3$ and $Y_3$ is $\Mat_3(\K)$.
\end{lemma}

\begin{proof}
Consider the nine elements
\[
 I_3,\ X_3,\ Y_3,\ X_3Y_3,\ Y_3X_3,\ Y_3^2,\ X_3Y_3^2,\
 Y_3X_3Y_3,\ Y_3X_3Y_3^2.
\]
We compute their coordinate determinant explicitly.  Put
\[
  q=t^3,
\]
and abbreviate $\alpha=\alpha(t)$, $\beta=\beta(t)$ and
$\gamma=\gamma(t)$.  Reorder the nine matrices as
\[
  X_3,\ X_3Y_3,\ X_3Y_3^2;
  \qquad I_3,\ Y_3,\ Y_3^2;
  \qquad Y_3X_3,\ Y_3X_3Y_3,\ Y_3X_3Y_3^2.
\]
This permutation of the columns is even, so it does not change the sign of the
coordinate determinant.  Relative to the standard matrix-unit basis
\[
 E_{11},E_{12},E_{13},E_{21},E_{22},E_{23},E_{31},E_{32},E_{33},
\]
the first three matrices are supported in the first row and have coordinate
matrix
\[
 A=
 \begin{pmatrix}
  \gamma&q\beta&q\alpha\\
  \alpha&\gamma&q\beta\\
  \beta&\alpha&\gamma
 \end{pmatrix}.
\]
The last three matrices are supported in the third row and have coordinate
matrix
\[
 B=
 \begin{pmatrix}
  q\gamma&q^2\beta&q^2\alpha\\
  q\alpha&q\gamma&q^2\beta\\
  q\beta&q\alpha&q\gamma
 \end{pmatrix}
 =qA.
\]
Finally, the second rows of $I_3,Y_3,Y_3^2$ have coordinate matrix
\[
 C=
 \begin{pmatrix}
 0&0&q\\
 1&0&0\\
 0&1&0
 \end{pmatrix},
 \qquad \det C=q.
\]
With the rows grouped according to the first, second and third matrix rows, the
full $9\times9$ coordinate matrix has block form
\[
 \mathcal M=
 \begin{pmatrix}
   A&I_3&0\\
   0&C&0\\
   0&Q&B
 \end{pmatrix}
\]
for a certain $3\times3$ matrix $Q$ (its entries are irrelevant for the
determinant).  Expanding first in the first three columns, and then using the
block lower triangular form of the remaining matrix, gives
\[
 \det\mathcal M=(\det A)(\det C)(\det B)
 =q^4(\det A)^2.
\]
A direct $3\times3$ determinant calculation gives
\begin{align*}
 \det A
 &=\gamma
   \begin{vmatrix}\gamma&q\beta\\ \alpha&\gamma\end{vmatrix}
 -q\beta
   \begin{vmatrix}\alpha&q\beta\\ \beta&\gamma\end{vmatrix}
 +q\alpha
   \begin{vmatrix}\alpha&\gamma\\ \beta&\alpha\end{vmatrix}\\
 &=\gamma^3+q\alpha^3+q^2\beta^3-3q\alpha\beta\gamma.
\end{align*}
Consequently
\begin{equation}\label{eq:f3detexact}
 \det\mathcal M
 =t^{12}
 \left(
   \gamma(t)^3+t^3\alpha(t)^3+t^6\beta(t)^3
   -3t^3\alpha(t)\beta(t)\gamma(t)
 \right)^2.
\end{equation}
Now $\gamma(t)\in t^3\R$, $\beta(t)\in\R$, and $\alpha(t)$ is a unit.
Therefore the expression inside the parentheses in
\eqref{eq:f3detexact} is
\[
 t^3\alpha(t)^3+O(t^6).
\]
Thus
\[
 \det\mathcal M
 =t^{18}\alpha(t)^6+O(t^{21}).
\]
Since
\[
 \alpha(t)=a_0(\sigma-1)+O(t^3),
\]
the actual first nonzero $t$-adic term is
\[
  [a_0(\sigma-1)]^6t^{18},
\]
which is nonzero because $a_0\neq0$ and $\sigma\neq1$. Hence the nine matrices
are $\K$-linearly independent. Since $\dim_{\K}\Mat_3(\K)=9$, they form a
basis of $\Mat_3(\K)$, and the $\K$-algebra generated by $X_3,Y_3$ is all of
$\Mat_3(\K)$. 
\end{proof}

Let $L_3=\im(f_3)$. Since $X_3,Y_3\in\Mat_3(\R)$, $L_3\subseteq M_{3}(\R)$.  $\R={\mathbb C}[[t]]$ is Noetherian,
 $L_{3}$ is a finitely generated $\R$-module. It is torsion free and Lemma \ref{lem:f3rank} shows that its rank is $9$.
Modulo $t$ we have
\[
  x^2=y^3=xyx=xy^2x=0,
  \qquad
  y^2x+yxy+xy^2=0.
\]
A direct reduction gives the spanning set
\[
  1,\ x,\ y,\ xy,\ yx,\ y^2,\ yxy,\ y^2x,\ y^2xy
\]
for $L_3/tL_3$. Hence every word of sufficiently large positive degree vanishes
in $L_3/tL_3$; in particular, the ideal generated by $x$ and $y$ is nilpotent.
Thus $L_3/tL_3$ is local.

\subsection{The $2\times2$ component}

Let
\[
  q_2(z)=z^3-\sigma t^3,
  \qquad
  B_2=\K[z]/(q_2(z)),
  \qquad
  u_2=z+(q_2(z)),
\]
and put
\[
  A_2'=\Mat_2(B_2).
\]
Thus $u_2^3=\sigma t^3$.

Set
\begin{equation}\label{eq:cdef}
  c(t)=\frac{\gamma(t)^2}{4}+\frac{1-\sigma}{3}t^3.
\end{equation}
Because $\gamma(t)\in t^3\R$ and $\sigma\neq1$,
\[
  c(t)=\frac{1-\sigma}{3}t^3+O(t^6)\neq0.
\]
Define
\[
  X_2=
  \begin{pmatrix}
    \gamma(t)/2&1\\
    c(t)&\gamma(t)/2
  \end{pmatrix},
  \qquad
  Y_2=
  \begin{pmatrix}
    u_2&0\\
    0&\omega u_2
  \end{pmatrix}.
\]
Let
\[
  f_2:\R\langle x,y\rangle\longrightarrow A_2'
\]
be defined by $f_2(x)=X_2$, $f_2(y)=Y_2$.

Put
\[
  Q=X_2-\frac{\gamma(t)}2I_2
  =\begin{pmatrix}0&1\\c(t)&0\end{pmatrix}.
\]
Then
\[
  Q^2=c(t)I_2,
  \qquad
  Y_2^3=\sigma t^3I_2.
\]
By the definition of $c(t)$,
\[
  Q^2
  =\left(\frac{\gamma(t)^2}{4}+\frac{1-\sigma}{3}t^3\right)I_2,
\]
which is equivalent to $f_2(\Phi_1)=0$.

For the second relation, the off-diagonal entries of
\[
  Y_2^2X_2+Y_2X_2Y_2+X_2Y_2^2
\]
are multiples of
\[
  1+\omega+\omega^2=0,
\]
while its diagonal part is $(3/2)\gamma(t)Y_2^2$. Therefore
\[
  Y_2^2X_2+Y_2X_2Y_2+X_2Y_2^2
  -\frac32\gamma(t)Y_2^2=0.
\]
Also, because $Y_2^3=\sigma t^3I_2$,
\[
  g(Y_2)
  =g_0(\sigma t^3)I_2
   +g_1(\sigma t^3)Y_2
   +g_2(\sigma t^3)Y_2^2.
\]
Hence $f_2(\Phi_2)=0$.

\begin{lemma}\label{lem:f2rank}
The $\K$-algebra generated by $X_2$ and $Y_2$ is $A_2'$.
\end{lemma}

\begin{proof}
Since $c(t)\neq0$ in $\K$, we may form
\[
  Z=c(t)^{-1}QY_2Q
   =\begin{pmatrix}\omega u_2&0\\0&u_2\end{pmatrix}.
\]
The two diagonal matrices $Y_2$ and $Z$ have linearly independent diagonal
patterns because $\omega\neq1$. Hence their $\K$-linear combinations give
\[
  \begin{pmatrix}u_2&0\\0&0\end{pmatrix},
  \qquad
  \begin{pmatrix}0&0\\0&u_2\end{pmatrix}.
\]
Taking powers and using $u_2^3=\sigma t^3\in\K^\times$ yields all diagonal
matrix units with coefficients in $B_2$. Multiplication by $Q$ gives the two
off-diagonal matrix units, again with arbitrary coefficients in $B_2$.
Therefore the generated algebra is $\Mat_2(B_2)$. This finishes the proof.
\end{proof}

Let $L_2=\im(f_2)$. Recall that $X_2,Y_2$ lie in
$\Mat_2(\R[u_2])$. Since $L_{2}$ is a finitely generated $\R$-module and $\R$ is Noetherian, $L_2$ is a finite torsion-free $\R$-module, and
Lemma~\ref{lem:f2rank} shows that it has rank $12$.
Modulo $t\im (f_{2})$, the relations become
\begin{equation}\label{eq:L2special1}
  X_{2}^2=0,
  \qquad
  Y_{2}^3=0,
  \qquad
  Y_{2}^2X_{2}+Y_{2}X_{2}Y_{2}+X_{2}Y_{2}^2=0.
\end{equation}

Let
\[B ={\mathbb C}\langle X,Y\rangle/(X^{2},Y^{3},Y^{2}X +YXY +XY^{2}).\]
A homogeneous Diamond-lemma computation  with the usual degree-lexicographic monomial order
with  $X>Y$, $\deg X=\deg Y=1$ gives 
 the initial reduction rules:

\[X^{2}\rightarrow 0, Y^{3}\rightarrow 0, XY^{2}\rightarrow -Y^{2}X-YXY.\]
 After resolving the ambiguities we also obtain the reduction rules
\[XYXY\rightarrow YXYX, YXYXY\rightarrow Y^{2}XYX.\]
 All the remaining ambiguities are resolved.

 The list of standard monomials is:  
\[1, X,Y,XY,YX,Y^{2}, XYX, YXY, Y^{2}X, YXYX,Y^{2}XY,Y^{2}XYX.\]

In particular, $\dim _{C} B = 12$, and all words of degree at least $6$ vanish. 
 Relations (\ref{eq:L2special1})  
induce a surjective homomorphism from $B$ to $ L_{2}/tL_{2}$. Since $L_{2}$ is free of rank $12$, one has
$\dim _{C}L_{2}/tL_{2} = 12$, so this homomorphism is an isomorphism.  In particular, $L_{2}/tL_{2}$ is local and its
  ideal generated by $X_{2}+ t\im (f_{2})$ and $Y_{2}+ t\im (f_{2})$ is nilpotent.

\subsection{The commutative degree-six component}\label{34}

Define
\begin{equation}\label{eq:Ddef}
  D(y)=g_0(\sigma t^3)+g_1(\sigma t^3)y+g_2(\sigma t^3)y^2-g(y).
\end{equation}
The second relation in the commutative component will give
\[
  x-\frac{\gamma(t)}2=\frac{D(y)}{3y^2}.
\]
Combining this with \eqref{eq:Phi1} leads to the polynomial
\begin{equation}\label{eq:Hdef}
  H(y)
  =D(y)^2
   -9y^4\left(
      \frac{\gamma(t)^2}{4}+\frac{t^3-y^3}{3}
    \right)
  \in\R[y].
\end{equation}

Modulo $t$ we have
\[
  D(y)\equiv-g(y)\pmod t,
  \qquad
  \gamma(t)\equiv0\pmod t,
\]
so
\begin{equation}\label{eq:Hmodt}
  H(y)\equiv g(y)^2+3y^7
  =y^6\bigl(a_0^2+O(y)\bigr)
  \pmod t.
\end{equation}
The factor multiplying $y^6$ has nonzero constant term. After multiplying
$H$ by a nonzero constant if necessary, we may assume that $H$ is monic. (A detailed explanation of why we can assume that $H(y)$ is monic is as follows:  
Let $d=\deg g \geq 3$.

If $d\geq 4$, then $D(y)$ has the same leading degree $d$ as $g(y)$, with leading coefficient equal 
 to the negative of the leading coefficient of $g$. Hence $D(y)^{2}$ has degree $2d\geq 8$, while the 
second term in 
\[H(y)=D(y)^{2}-9y^{4} ({\frac {{ \gamma(t)}^{2}}{4}}
 +{\frac { t^{3}-y^{3}}
3}) \]
has degree at most $7$. Therefore the leading coefficient of $H$ is a nonzero complex constant.

 If $d=3$, then the term $+3y^{7}$ coming from the second summand is the highest-degree term, 
so the leading coefficient is $3$.
 Thus in all cases the leading coefficient belongs to ${\C}^{\times}$, and multiplication by a nonzero 
complex constant indeed makes $H$ monic.)

 Since \(R=\mathbb C[[t]]\) is a complete local ring, we may use the factorisation form of
Hensel's lemma; see \cite{Mats}, 
Theorem~8.3.
By Hensel's lemma, there is a monic
factor
\[
  h_1(y)\in\R[y],\qquad \deg h_1=6,
\]
such that
\begin{equation}\label{eq:h1modt}
  h_1(y)\equiv y^6\pmod t
\end{equation}
and $h_1(y)$ divides $H(y)$ in $\R[y]$. (A detailed proof of why $h_{1}$ has degree $6$ is as follows:  
  Since
\[H(y) \equiv  y^{6}U(y) \mod t,
U(0) = a_{0}^{2}\neq  0,\]
we have $\gcd(y ^{6},U) = 1$. After normalising $H$ to be monic, Hensel’s lemma yields monic
polynomials $h_{1},h_{2} \in {\mathcal R}[y]$ such that
\[ H =h_{1}h_{2},\]
\[h_{1} \equiv  y^{6} \mod t,
h_{2}\equiv  U \mod t .\]
Since $h_1$ is monic and reduces to the monic polynomial $y^6$, necessarily
$\deg h_{1} = 6$.)

$ $
   
 Observe that  $h_1(0)\neq0$: indeed $H(0)=g_0(\sigma t^3)^2\neq0$, while the
complementary Hensel factor has a unit constant term by
\eqref{eq:Hmodt}.

$ $
Set
\[
  B_1=\K[z]/(h_1(z)),
  \qquad
  u_1=z+(h_1(z)),
  \qquad
  A_1'=B_1.
\]
Since $h_1(0)\neq0$, $u_1$ is invertible in $B_1$
 (so the expression $u_{1}^{-2}D(u_{1})$ used in the definition of $X_{1}$ is legitimate).
 Define
\[
  Y_1=u_1,
\]
and
\begin{equation}\label{eq:X1def}
  X_1
  =\frac{\gamma(t)}2
   +\frac{1}{3}u_1^{-2}D(u_1).
\end{equation}
Let
\[
  f_1:\R\langle x,y\rangle\longrightarrow A_1'
\]
be the homomorphism given by $f_1(x)=X_1$, $f_1(y)=Y_1$.

Since $h_1$ divides $H$, equation~\eqref{eq:Hdef} gives in $B_1$
\[
  \left(\frac{D(u_1)}{3u_1^2}\right)^2
  =\frac{\gamma(t)^2}{4}+\frac{t^3-u_1^3}{3}.
\]
Using \eqref{eq:X1def}, this is exactly
\[
  X_1^2-\gamma(t)X_1+\frac{u_1^3-t^3}{3}=0,
\]
so $f_1(\Phi_1)=0$.

Since $B_1$ is commutative,
\[
  Y_1^2X_1+Y_1X_1Y_1+X_1Y_1^2=3Y_1^2X_1.
\]
By \eqref{eq:X1def},
\[
  3Y_1^2X_1-\frac32\gamma(t)Y_1^2=D(Y_1).
\]
Using \eqref{eq:Ddef}, we obtain $f_1(\Phi_2)=0$.

Let $L_1=\im(f_1)$. Since $Y_1=u_1$, the generic fibre of $L_1$ is all of
$B_1$, so $L_1$ has rank $6$ over $\R$. It is finite over $\R$: the element
$Y_1$ is integral by $h_1(Y_1)=0$, and $X_1$ is integral over $\R[Y_1]$ by
\eqref{eq:Phi1}. Thus $L_1$ is a finite torsion-free, hence free, $\R$-module of
rank $6$.

Modulo $tL_{1}=t\im(f_1)$, equation~\eqref{eq:h1modt} gives $Y_{1}^6=0$, while
\eqref{eq:Phi1} gives
 \[ X_1^2=-\frac13Y_1^3,
\qquad\text{hence}\qquad X_1^4=\frac19Y_1^6=0.
\] 
Because $L_1/tL_1$ is commutative, the ideal generated by $X_{1}+t\im(f_1)$ and $Y_{1}+t\im(f_1)$ is
nilpotent. Hence $L_1/tL_1$ is local.

\subsection{The full homomorphism and its rank}\label{35}

Put
\[
  A'=A_1'\oplus A_2'\oplus A_3'
\]
and define
\[
  f=(f_1,f_2,f_3):\R\langle x,y\rangle\longrightarrow A'.
\]
 We have $f_{i}(x)=X_{i}$ and $f_{i}(y)=Y_{i}$ for $i=1, 2, 3$. 
Moreover, $f(x)=(X_{1}, X_{2}, X_{3})$, $f(y)=(Y_{1}, Y_{2}, Y_{3})$. 
By construction,
\[
  f(\Phi_1)=f(\Phi_2)=0.
\]

We next verify the rank condition required by  Theorem \ref{New}.
Over $\K$, Lemma~\ref{lem:f3rank} gives
\[
  \K\langle X_3,Y_3\rangle=\Mat_3(\K),
\]
 hence \[\dim_{\K}\K\langle X_3,Y_3\rangle=9.\]
 Lemma~\ref{lem:f2rank} gives
\[
  \K\langle X_2,Y_2\rangle=\Mat_2(B_2).
\]

 Hence \[\dim_{\K}\K\langle X_2,Y_2\rangle=12.\]

The first component is $B_1$, which is commutative and generated by $Y_1$.
Furthermore,
\[
  B_2\cong\K^{\oplus3},
\]
because $z^3-\sigma t^3$ has three distinct roots in $\K$.

 Since $h_{1}(z)$ is monic of degree $6$, $  \dim _{\K}\K\langle X_1,Y_1\rangle=6$. 
The boundedness hypothesis also holds. Indeed, $L_1$ is a finite $\R$-module in
the finite-dimensional $\K$-space $A_1'$, while
\[
  L_2\subseteq\Mat_2(\R[u_2]),
  \qquad
  L_3\subseteq\Mat_3(\R).
\]
Thus, for some $\delta\geq0$,
\[
  \im(f)\subseteq t^{-\delta}M\{t\}.
\]
Theorem~\ref{New} applies with
\[
  (n_1,d_1)=(1,6),\qquad
  (n_2,d_2)=(2,3),\qquad
  (n_3,d_3)=(3,1).
\]
Therefore
\[
  L:=\im(f)
\]
is a free $\R$-module of rank
\[
  1^2\cdot6+2^2\cdot3+3^2\cdot1=27,
\]
and $L/tL$ admits a flat deformation to 
\[
  \C^{\oplus6}\oplus\Mat_2(\C)^{\oplus3}\oplus\Mat_3(\C).
\]

\subsection{Identification of the special fibre}\label{36}

It remains to identify $L/tL$ with the contraction algebra $N$.
The two relations \eqref{eq:Phi1} and \eqref{eq:Phi2} reduce modulo $t$ to
\[
  3x^2+y^3=0,
  \qquad
  y^2x+yxy+xy^2+g(y)=0.
\]
Hence there is a surjective homomorphism
\[
  \C\langle x,y\rangle/I\twoheadrightarrow L/tL.
\]
We now show that the ideal generated by the images of $x$ and $y$ in $L/tL$ is nilpotent; 
this implies that the homomorphism factors through the completion $N$. 

Let $L_i=\im(f_i)$. The natural inclusion
\[
  L\hookrightarrow L_1\oplus L_2\oplus L_3
\]
becomes an isomorphism after tensoring with $\K$, because the generic image is
the full direct product. Hence the cokernel is a finite torsion $\R$-module.
Therefore there exists $c\geq0$ such that
\begin{equation}\label{eq:conductor}
  t^c(L_1\oplus L_2\oplus L_3)\subseteq L.
\end{equation}
For each $i$, the ideal generated by the images of $x,y$ in $L_i/tL_i$ is
nilpotent, as proved above. Choose $e$ large enough that every word of length
$e$ in $x,y$ maps into $tL_i$ for all three $i$.

Any word of length $e(c+1)$ can be divided into $c+1$ consecutive blocks of
length $e$. Its image in each $L_i$ therefore belongs to $t^{c+1}L_i$.
By \eqref{eq:conductor},
\[
  t^{c+1}(L_1\oplus L_2\oplus L_3)\subseteq tL.
\]
Thus every sufficiently long word in the images of $x,y$ belongs to $tL$.
Consequently the ideal generated by $x,y$ in $L/tL$ is nilpotent, so $L/tL$ is
local.

The map from $\C\langle x,y\rangle/I$ to $L/tL$ therefore extends to a
surjection
\[
  N\twoheadrightarrow L/tL.
\]
Both algebras have dimension $27$ over $\C$, hence this surjection is an
isomorphism. We have proved the following.

\begin{theorem}
Let $p(y)\in\C[y]$ satisfy $p(y)\neq0$ and $\ord(p)=4$, and let $N$ be the
completed Jacobi algebra associated with the potential
\[
  x^3+xy^3+p(y).
\]
 Suppose that $N$ has dimension $27$ and that the ideal generated by $x$ and $y$ is nilpotent (this happens for example when $N$ is a contraction algebra by \cite{MW}). 
Then $N$ admits a flat deformation to
\[
  \boxed{
  \C^{\oplus6}\oplus\Mat_2(\C)^{\oplus3}\oplus\Mat_3(\C)}.
\]
\end{theorem}

\section{ Contraction algebras with $\ord(p)>4$ and $\ord(p)-1$ not divisible by $3$}

\subsection{Hypotheses}
 We recall the method which will be used in this section. 
Let
\[
  f:\R\langle x_1,\dots,x_\rho\rangle\longrightarrow A'
\]
be a homomorphism of unital $\R$-algebras. We impose the following conditions.

\begin{enumerate}[label=\textup{(\arabic*)}]
\item $f(1)=1_{A'}$.
\item There exist $m_1,\dots,m_N\in\C\langle x_1,\dots,x_\rho\rangle$
      such that $f(m_1),\dots,f(m_N)$ are linearly independent over $\R$.
\item There exists $\delta\geq 0$ such that
      \[
        \im(f)\subseteq t^{-\delta}M\{t\}.
      \]
\end{enumerate}

The next statement is the version of the method that we shall use (which is Theorem \ref{thm:nonsquarefree}).  
 
\begin{theorem}\label{11}
Assume that $f$ satisfies the hypotheses above. Set
\[
  N_0= \im(f)/t\im(f).
\]
 There is a flat deformation from $N_0$ to $A$.
\end{theorem}

 Let $\ord(p)=n$, so $\ord(g)=n-1$, where $g(y)=p'(y)$. Suppose that $3$ does not divide $n-1$.

Our target algebra is
\[
  A=\C^{\oplus 7}\oplus\Mat_2(\C)^{\oplus n-1}\oplus\Mat_3(\C).
\] 
 We now construct a homomorphism to which Theorem \ref{11} 
applies. Namely, we will construct the three components and then apply Theorem \ref{New} (and hence Theorem \ref{11}).

\subsection{Preliminary notation and the common deformed relations}\label{1}

 Set $m=\deg (g(y))$, $e_{n-1}e_{m}\neq 0$, and write
\[g(y)=\sum_{i=n-1}^{m}e_{i}y^{i}.\]

 Put 
\[k = \lfloor {\frac {n-1}{3}}\rfloor, m'=\lfloor {\frac {m}{3}}\rfloor.\]

Define  \[\gamma(t)= 2(\sum_{i=k}^{m'}e_{3i+2}t^{3i}).\label{eq:gamma}\]
We shall find $f$ which annihilates 
\begin{equation}\label{eq:Phi1drugi}
  \Phi_1
  =x^2-\gamma(t)x+\frac{1}{3}(y^3-t^3)
\end{equation}
and
\begin{equation}\label{eq:Phi2drugi}
\begin{split}
  \Phi_2={}&y^2x+yxy+xy^2+g(y)
  +t^{3m+3}-\frac{3}{2}\gamma(t)y^2.
\end{split}
\end{equation}
Modulo $t$, these are equivalent to the two Jacobi relations. 

Notice that $\Phi _{1}=0$ is equivalent to $\Phi_{3}=0$ where 
\begin{equation}\label{eq:Phi3}
  \Phi_3
  =(x-{\frac {\gamma(t)}2})^{2}-{\frac {\gamma(t)^{2}}4}+\frac{1}{3}(y^3-t^3).
\end{equation}
 Moreover, $\Phi _{2}=0$ is equivalent to $\Phi _{4}=0$ where 
\begin{equation}\label{eq:Phi4}
\begin{split}
  \Phi_4={}&y^2(x-{\frac {\gamma(t)}2})+y(x-{\frac {\gamma(t)}2})y+(x-{\frac {\gamma(t)}2})y^2+g(y)
  +t^{3m+3}.
\end{split}
\end{equation}

\subsection{The $3\times3$ component}
 This subsection is almost identical to  Subsection \ref{3by3}.
 We only take slightly different elements $\alpha (t), \beta (t), \gamma (t)$, but the remaining part of the construction and its proofs are essentially the same.

 Let the notation be as in Section \ref{M1}. 
 Set $e_{j}=0$ for all $j\notin [n-1,m]$.
  Define 

\[ \alpha(t)= -t^{3m}-t^{-3}(\sum_{i=k}^{m'}e_{3i}t^{3i}),\label{eq:alpha}\\\]
 \[ \beta(t)= -t^{-3}(\sum_{i=k}^{m'}e_{3i+1}t^{3i}),\label{eq:beta}\\ \]
  \[\gamma(t)= 2(\sum_{i=k}^{m'}e_{3i+2}t^{3i}).\label{eq:gamma}\]
  
Notice that 
\[
  \gamma\in t^{3}\R,
  \qquad \alpha,\beta\in \R.
\]

Put
\[
  A_3'=\Mat_3(\K).
\]
 Define
\[
  X_3=
  \begin{pmatrix}
    \gamma(t)&\alpha(t)&\beta(t)\\
    0&0&0\\
    0&0&0
  \end{pmatrix},
  \qquad
  Y_3=
  \begin{pmatrix}
    0&1&0\\
    0&0&1\\
    t^3 &0&0
  \end{pmatrix}.
\] 
Let
\[
  f_3:\R\langle x,y\rangle\longrightarrow A_3'
\]
be the unital homomorphism defined by $f_3(x)=X_3$ and $f_3(y)=Y_3$.

We have
\[
  Y_3^3=t^3I_3,
  \qquad
  X_3^2=\gamma(t)X_3.
\]
Hence $f_3(\Phi_1)=0$.

A direct multiplication gives
\begin{equation}\label{eq:f3mixed}
  Y_3^2X_3+X_3Y_3^2+Y_3X_3Y_3
  =t^3\alpha(t)I_3+t^3\beta(t)Y_3+\gamma(t)Y_3^2.
\end{equation}
Since $Y_3^3=t^3I_3$, 
\[
  g(Y_3)=\sum_{i=k}^{m'}(e_{3i}t^{3i} I_3+e_{3i+1}t^{3i}Y_{3}+e_{3i+2}t^{3i}Y_{3}^{2}).
\]
 
Using the definitions of $\alpha (t),\beta (t),\gamma (t)$, the coefficients of
$I_3,Y_3,Y_3^2$ in $f_3(\Phi_2)$ all vanish. Thus
\[
  f_3(\Phi_2)=0.
\]

For later reference, some useful products are
\begin{align*}
X_3Y_3&=
\begin{pmatrix}
 t^3\beta&\gamma&\alpha\\0&0&0\\0&0&0
\end{pmatrix},
&
Y_3X_3&=
\begin{pmatrix}
0&0&0\\0&0&0\\t^3\gamma&t^3\alpha&t^3\beta
\end{pmatrix},\\[2mm]
X_3Y_3^2&=
\begin{pmatrix}
 t^3\alpha&t^3\beta&\gamma\\0&0&0\\0&0&0
\end{pmatrix},
&
Y_3X_3Y_3&=
\begin{pmatrix}
0&0&0\\0&0&0\\t^6\beta&t^3\gamma&t^3\alpha
\end{pmatrix},\\[2mm]
Y_3^2X_3&=
\begin{pmatrix}
0&0&0\\t^3\gamma&t^3\alpha&t^3\beta\\0&0&0
\end{pmatrix}.
\end{align*}
Also
\begin{align*}
  X_3Y_3X_3&=t^3\beta(t)X_3,\\
  X_3Y_3^2X_3&=t^3\alpha(t)X_3,\\
  X_3^3&=\gamma(t)^2X_3.
\end{align*}
Two of the longer products that are useful for checking the original matrix
list are
\[
Y_3^2X_3Y_3X_3=
\begin{pmatrix}
0&0&0\\
 t^6\beta\gamma&t^6\alpha\beta&t^6\beta^2\\
0&0&0
\end{pmatrix}
\]
and
\[
Y_3^2X_3Y_3^2=
\begin{pmatrix}
0&0&0\\
 t^6\alpha&t^6\beta&t^3\gamma\\
0&0&0
\end{pmatrix}.
\]

\begin{lemma}\label{lem:f3rank2}
The $\K$-algebra generated by $X_3$ and $Y_3$ is $\Mat_3(\K)$. Denote $L_{3}=\im (f_{3})$. 
Then the ideal of $L_{3}/tL_{3}$ generated by $X_{3}+tL_{3}$ and $Y_{3}+ tL_{3}$ is nilpotent.
\end{lemma}

\begin{proof} The proof is the same as the proof of Lemma \ref{lem:f3rank}. 
The fact that $\alpha (t), \beta (t)\in  {\mathbb C}[[t]]$ does not influence the proof. Indeed, recall that the determinant of the matrix  $M$ used in the proof of Lemma \ref{lem:f3rank} is 
\begin{equation}\label{eq:f3detexactnew}
 \det\mathcal M
 =t^{12}
 \left(
   \gamma(t)^3+t^3\alpha(t)^3+t^6\beta(t)^3
   -3t^3\alpha(t)\beta(t)\gamma(t)
 \right)^2=t^{12}P(t)^{2}.
\end{equation} Put $r=n-1$.
\begin{itemize}
\item If $r=3k+1$, then $e_r\ne0$ and
\[
 \beta(t)=-e_r t^{3k-3}+O(t^{3k}),\qquad
 \alpha(t)=O(t^{3k}),\qquad \gamma(t)=O(t^{3k}).
\]
Hence
\[
 P(t)=-e_r^3t^{9k-3}+O(t^{9k}),
\]
so $P(t)\ne0$.
\item If $r=3k+2$, then
\[
 \gamma(t)=2e_r t^{3k}+O(t^{3k+3}),\qquad
 \alpha(t),\beta(t)=O(t^{3k}),
\]
and therefore
\[
 P(t)=8e_r^3t^{9k}+O(t^{9k+3})\ne0.
\]
\end{itemize}
  Hence  the determinant of $M$ is nonzero. 
 Recall that  $L_3=\im (f_3)$. The $K$-algebra generated by $X_3$ and $Y_3$ is $M_3(K)$.
    Moreover, the ideal of  $L_3/tL_3$ generated by $X_3+tL_3$ and $Y_3+tL_3$ 
    is nilpotent (by the same argument as at the end of Section \ref{Section3}). 
 Hence $L_3/tL_3$ is local.
\end{proof}

\subsection{The $2\times2$ component}\label{dwa}

 Let $\omega \in {\mathbb C}$ be such that $\omega \neq 1$ and ${\omega }^{3}=1$. 
 
Set
\[
T=t^{3m+3},
\qquad
Q_2(z)=g(z)+T+\omega\bigl(g(\omega z)+T\bigr).
\]

 Observe that $z^{n-1}$ divides $Q_{2}(z)$ modulo $t$, and $z^{n}$ does not divide $Q_{2}(z)$ modulo $t$.   We have
\[
 Q_2(z)\bmod t=z^{n-1}U(z),\qquad U(0)\ne0,
\]
after multiplying $Q_2$ by a nonzero complex scalar if necessary to make it monic.  Since $z^{n-1}$ and $U(z)$ are coprime, Hensel gives a monic polynomial $q_{2}(z)\in  {\mathbb C}[[t]][z]$ of 
degree $n-1$ (the same as the order of $g(y)$) such that $q_{2}(z)$ divides $Q_{2}(z)$ in ${\mathbb C}[[t]][z]$ and 
\[q_{2}(z)=z^{n-1}+\sum_{i=0}^{n-2}c_{i}z^{i}\] and $c_{i}\in t{\mathbb C}[[t]]$. 
Let
\[
  B_2=\K[z]/(q_2(z)),
  \qquad
  u_2=z+(q_2(z)),
\]
and put

\[
  A_2'=\Mat_2(B_2).
\]
Since  \(q_2\mid Q_2\) and \(u=u_2\), in \(B_2\) we have
\[
g(u)+T+\omega\bigl(g(\omega u)+T\bigr)=0.
\]

Let
\[
Y_2=
\begin{pmatrix}
u&0\\
0&\omega u
\end{pmatrix}.
\]
 Observe that  $Q_{2}(0) = (1+\omega )T \neq  0$. Since $q_{2} | Q_{2}$, $q_{2}(0) \neq  0$ and $u$ is an invertible element in $B_{2}$.
 Denote
\[
\boxed{
d=-\,\frac{g(u)+T}{3u^2}
=
\frac{g(\omega u)+T}{3\omega^2u^2}.
}
\]

 From $Q_{2}(u)=0$,
\[T=-{\frac 1{1+\omega }} \sum_{i=n-1}^{m}e_{i}(1+\omega ^{i+1})u_{i}.\]
Substituting into the formula for $d$ gives 
\[d=-{\frac {\omega }{3(1+\omega )}}\sum_{i=n-1}^{m}e_{i}(1-\omega ^{i})u^{i-2}\in u^{n-3}R[u].\]
Thus
\[
g(u)+T=-3du^2,
\qquad
g(\omega u)+T=3d\omega^2u^2.
\]

Define
\[
c=
\frac{\gamma(t)^2}{4}
-d^2
+\frac{t^3-u^3}{3},
\]
and
\[
X_2=
\begin{pmatrix}
\frac{\gamma}{2}+d&1\\
c&\frac{\gamma}{2}-d
\end{pmatrix}.
\]

Explicitly define $f_{2} : \R\langle x, y\rangle \rightarrow   M_{2}(B_{2})$ by $x\rightarrow X_{2}$, $y \rightarrow Y_{2}$.

Put
\[
Q=X_2-\frac{\gamma}{2}I
=
\begin{pmatrix}
d&1\\
c&-d
\end{pmatrix}.
\]
Then
\[
Q^2=(d^2+c)I
=
\left(
\frac{\gamma^2}{4}
+
\frac{t^3-u^3}{3}
\right)I.
\]
Therefore
\[
Q^2-\frac{\gamma^2}{4}I
+
\frac13\bigl(Y_2^3-t^3I\bigr)
=0,
\]
which is equivalent to \(f_2(\Phi_1)=0\).

For the second relation,
\[
Y_2^2Q+Y_2QY_2+QY_2^2
\]
has diagonal part
\[
\operatorname{diag}(3du^2,-3d\omega^2u^2),
\]
and the off-diagonal part vanishes because 
\[
1+\omega+\omega^2=0.
\]
Since
\[
g(Y_2)+TI_2
=
\begin{pmatrix}
g(u)+T&0\\
0&g(\omega u)+T
\end{pmatrix}
=
\begin{pmatrix}
-3du^2&0\\
0&3d\omega^2u^2
\end{pmatrix},
\]
the diagonal terms cancel, and hence
\[
f_2(\Phi_4)=0.
\] This implies $f_{2}(\Phi _{2})=0$.
\begin{lemma}\label{62}
 The element  $c$ is not a zero divisor in $B_{2}$. 
\end{lemma}
\begin{proof}  Take the Puiseux extension ${\mathbb C}((v))$ with $v^{e}=t$ in which $q_{2}(z)$ splits into linear factors.
 Let $u'$ be a root of   
 $q_{2}(z)$. Since $q_{2}(z)$ divides $Q_{2}(z)$, $u'$ is also a root of $Q_{2}(y)$. We can write
 \[ u'=\sum_{i=i_{0} }^{\infty} \xi_{i}v^{i},\]
where $\xi _{i}\in {\mathbb C}$, $\xi _{i_{0}}\neq 0$.  Since $q_{2}(u')=0$ and  $q_{2}(z)$ is monic, $i_{0}>0$ ( because $q_{2}(z)-z^{n-1}\in t{\mathbb C}[[t]]$). 
 
Observe that  $Q_{2}(u')=0$. Since  $1+\omega^n\neq 0$ because $\omega^n\in\{1,\omega,\omega^2\}$ is never $-1$, the lowest power of $y$ in $Q_{2}(y)$ is $n-1$ (since $e_{n-1}\neq 0$).
 By examining the lowest power of $v$ which appears in summands of  $Q_{2}(u')$ we obtain   
 $(1+\omega )t^{3m+3}+e_{n-1}(1+\omega^{n})\xi _{i_{0}}^{n-1}v^{i_{0}(n-1)}= (1+\omega )v^{e(3m+3)}+e_{n-1}(1+\omega^{n})\xi _{i_{0}}^{n-1}v^{i_{0}(n-1)}$ (recall that $t=v^{e}$).  Therefore cancellation of the lowest terms forces $(n-1)i_0=(3m+3)e$.

 Recall that \[
c=
\frac{\gamma(t)^2}{4}
-d^2
+\frac{t^3-u^3}{3}=\frac{\gamma(t)^2}{4}
-(-\,\frac{g(u)+T}{3u^2})^2
+\frac{t^3-u^3}{3}, 
\]
 
Suppose on the contrary that  $c$ is a zero divisor in $B_{2}$.  Then  there is a root $u'$ of $q_{2}(z)$ in ${\mathbb C}((v))$ such that 
  \[
0=\frac{\gamma(t)^2}{4}
-(-\,\frac{g(u')+T}{3u'^2})^2
+\frac{t^3-u'^3}{3},
\]

 Notice that $(u')^{2}\neq 0$, hence $(u_{2}')^{2}$ is invertible in the field ${\mathbb C}((v))$.    

 Recall that $\gamma (t)\in t^{3}\mathbb C[[t]]$, hence  $\gamma (t)^{2}\in t^{6}\mathbb C[[t]]$. 
 By multiplying the above equation by $9u'^{4}$ we obtain the equation:
  \[
0={9u'^4}\frac{\gamma(t)^2}{4}
-(\,{g(u')+T})^2
+{9u'^4}\frac{t^3-u'^3}{3},
\]

 By the assumptions of this section, $g(z)$ has order $n-1$, so it has a nonzero coefficient at $z^{n-1}$. Recall also that $T=t^{3m+3}$ and $m\geq n-1$. By  substituting $u'=\sum_{i=i_{0} }^{\infty} \xi_{i}v^{i}$ in this equation by examining the smallest power of $v$ appearing as a nonzero summand in ${g(u')+T}$
  we get that it is at least 
\[min((n-1)i_{0},  (3m+3)e).\]
 By the above, $min((n-1)\cdot i_{0},  (3m+3)e)=(n-1)i_{0}=(3m+3)e$. 
So the lowest power of $v$ in a nonzero summand  appearing at $({g(u')+T})^{2}$
 is  at least \[2(n-1)i_{0}=(6m+6)e.\]

 On the other hand the smallest power of $v$ appearing in the summand $9u'^{4}({\frac{\gamma(t)^2}{4}}+{\frac{t^3-u'^3}{3}})$ 
 is (since ${\gamma(t)^2}\in t^{6}{\mathbb C}[[t]]\subseteq t^{6e }{\mathbb C}[[v]]$) coming from $9u'^{4}{\frac{t^3-u'^3}{3}}$ and is \[min(3e+4i_{0}, 3i_{0}+4i_{0})\leq 7i_{0}.\]  
This is less than $2(n-1)i_{0}$ since $n-1\geq 4$. 
 Observe also that $3e+4i_{0}\leq 7i_{0}$ (since $n-1\leq m$). Therefore, the lowest power of $v$ appearing in the formula
\[{9u'^4}\frac{\gamma(t)^2}{4}
-(\,{g(u')+T})^2
+{9u'^4}\frac{t^3-u'^3}{3}\] is 
$3e+4i_{0}$ and is coming from $9u'^{4}t^{3}$.

 We see that the coefficient at the smallest power of $v$ occurring is nonzero. Therefore, 
$0\neq \frac{\gamma(t)^2}{4}
-(-\,\frac{g(u')+T}{3u'^2})^2
+\frac{t^3-u'^3}{3},$ a contradiction. Consequently, 
  $c$ is not a zero divisor in the finite-dimensional $\K$-algebra $B_{2}$. 
  \end{proof}

\begin{lemma}\label{lem:f2ranknew}
The $\K$-algebra generated by $X_2$ and $Y_2$ is $A_2'=\Mat_2(B_2)$.
\end{lemma}

\begin{proof} Define \[
  \mathcal A:=K\langle X_2,Y_2\rangle\subseteq M_2(B_2).
\]
Observe first that by the second deformed Jacobi relation (by looking at the diagonal entries in this relation) $LY_{2}^{2}\in \K\langle Y_{2}\rangle $ where:
\[
  L   =\begin{pmatrix}{\frac {\gamma}2}+d&0\\0&{\frac {\gamma}2}-d\end{pmatrix}.
\]
 
Hence
\[
{\tilde  Q}=Y_{2}^{2}(X_2-L)\in\mathcal A.
\]
Consequently
\[
 {\tilde Q}=
 \begin{pmatrix}
 0&u^2\\
 \omega^2 c u^2&0
 \end{pmatrix}.
\]
A direct multiplication gives
\[
{\tilde  Q}Y_2{\tilde Q}=c u^5
 \begin{pmatrix}1&0\\0&\omega^2\end{pmatrix},
 \qquad
 {\tilde Q}^2Y_2=c u^5
 \begin{pmatrix}\omega^2&0\\0&1\end{pmatrix}.\]

 Since
$\omega^2\ne1$, their diagonal patterns are linearly independent.  Hence the $K$-linear
combinations  of  $({\tilde  Q}Y_2{\tilde Q})Y_{2}^{j},   ({\tilde Q}^2Y_2) Y_{2}^{j}$  for $0\le j<n-1$ gives 
\[
   cu^{5+j}E_{11},\qquad cu^{5+j}E_{22}.
\]  Multiplication by $c$ is injective by
Lemma \ref{62}, and $u$ is a unit, so these elements generate a $2(n-1)$-dimensional diagonal
subspace, which must be the full diagonal subspace of $M_2(B_2)$.  Since all diagonal matrices
belong to the generated algebra, $L$ belongs to it, hence $X_2-L$ belongs to it; multiplying by
matrix units produces the full subspace of off-diagonal matrices.  

\end{proof} 

We will now show that the ideal generated by the images of $X_{2}$ and $Y_{2}$ in $\im(f_{2})/t(\im(f_{2}))$ is nilpotent. Let $L_2=\im(f_2)$.  
Since $X_2,Y_2$ lie in
$\Mat_2(\R[u_2])$, $L_2$ is a finite torsion-free $\R$-module, and
Lemma~\ref{lem:f2ranknew} shows that it has rank $4\cdot(n-1)$. 
 Recall that 
 $f_{2}(y^2x+yxy+xy^2+g(y))\in t\im (f_{2})$. 
 Hence, $X_{2}Y_{2}^{2}+Y_{2}X_{2}Y_{2}+Y_{2}^{2}X_{2}+g(Y_{2})\in t\im (f_{2})$.
 Therefore, modulo $t\im (f_{2})$, we have the relations
\begin{equation}
  3X_{2}^2+Y_{2}^3=0,
  \qquad
Y_{2}^{3(n-1)}=0, 
\qquad
   Y_{2}^2X_{2}+Y_{2}X_{2}Y_{2}+X_{2}Y_{2}^2+g(Y_{2})=0.
\end{equation}
  To see that $Y_{2}^{3(n-1)}=0$ consider the following  an annihilating polynomial for $Y_2$.  If $r=n-1$, then
\[
 P(z):=q_2(z)q_2(\omega z)q_2(\omega^2z)
\]
annihilates $Y_2=\operatorname{diag}(u,\omega u)$ (up to a harmless nonzero scalar normalization).  Since $q_2(z)\equiv z^r\pmod t$,
\[
 P(z)\equiv c_0z^{3r}\pmod t,\qquad c_0\in\C^\times.
\]
Thus $P(Y_2)=0$ gives $Y_2^{3r}\in tL_2$, which is exactly the required relation in $L_2/tL_2$.

Consider
\[
C=
\C\langle x,y\rangle/
(3x^2+y^3,\,
y^{3(n-1)},\,
y^2x+yxy+xy^2+g(y)
).
\]
This should be regarded as a $\C$-algebra.
 Assign weight $m+2$ to $x$ and weight $1$ to $y$, and assign also ordering $x>y$ (in the degree-lexicographic monomial order).
 
The defining relations give
\[
x^2=-{\frac13}y^3,
\qquad
xy^2=-y^2x-yxy-g(y).
\]
Reducing \(x^2y^2\) in two different ways gives
\[
\boxed{xyxy=yxyx+g(y)x-xg(y).}
\]

Thus one may use the reductions
\[
x^2\longrightarrow-\frac13y^3,
\]
\[
xy^2\longrightarrow-y^2x-yxy-g(y),
\]
\[
xyxy\longrightarrow  yxyx+g(y)x-xg(y),
\]
\[
y^{3(n-1)}\longrightarrow 0.
\]
Recall that we took the weighted degree-lexicographic order with
\[
\deg x=m+2,\qquad \deg y=1,\qquad x>y.
\]

Every irreducible word has one of the forms
\[
y^a,\qquad
y^ax,\qquad
y^axy,\qquad
y^axyx,
\qquad
0\le a<3(n-1).
\]
Hence \(C\) is finite-dimensional.

 Since $y^{3(n-1)}=0$, the reductions  do not decrease ordinary word length, and after performing the reduction we obtain words of length larger than $3(n-1)+3$ which do not appear in the above spanning set of irreducible words, hence all sufficiently long words are zero. Consequently, the ideal
 $(x,y)$
is nilpotent. Therefore \(C\) is local.
\subsection{The commutative degree-seven component}
 This subsection is very similar to Subsection \ref{34}, except that the reduction of  $H(y)$ modulo $t$ has now order $7$ instead of $6$.  As in the previous subsections we assume that $n> 4$ and $3$ does not divide $n-1$. We repeat the argument. 
Define
\begin{equation}\label{eq:Ddef2}
  D(y)=-t^{3m+3}-g(y).
\end{equation}
The second relation in the commutative component will give
\[
  x-\frac{\gamma(t)}2=\frac{D(y)}{3y^2}.
\]
Combining this with \eqref{eq:Phi1drugi} leads to the polynomial
\begin{equation}\label{eq:Hdef2}
  H(y)
  =D(y)^2
   -9y^4\left(
      \frac{\gamma(t)^2}{4}+\frac{t^3-y^3}{3}
    \right)
  \in\R[y].
\end{equation}

Modulo $t$ we have
\[
  D(y)\equiv-g(y)\pmod t,
  \qquad
  \gamma(t)\equiv0\pmod t,
\]
so
\begin{equation}\label{eq:HmodtNew}
  H(y)\equiv g(y)^2+3y^7
  =y^7\bigl(3+O(y)\bigr)
  \pmod t.
\end{equation}
The factor multiplying $y^7$ has nonzero constant term (which is $3$). After multiplying
$H$ by a nonzero constant if necessary, we may assume that $H$ is monic.
By Hensel's lemma, there is a monic 
factor
\[
  h_1(y)\in\R[y],\qquad \deg h_1=7,
\]
such that
\begin{equation}\label{eq:h1modt2}
  h_1(y)\equiv y^7\pmod t
\end{equation}
and $h_1(y)$ divides $H(y)$ in $\R[y]$.
Moreover, $h_1(0)\neq0$ (since $H(0)\neq0$ and $H(y)=h_1(y)h_2(y)$): indeed $H(0)=t^{6m+6}\neq 0$, while the
complementary Hensel factor has a unit constant term by
\eqref{eq:HmodtNew}.

Set
\[
  B_1=\K[z]/(h_1(z)),
  \qquad
  u_1=z+(h_1(z)),
  \qquad
  A_1'=B_1.
\]
Since $h_1(0)\neq0$, $u_1$ is invertible in $B_1$. Define
\[
  Y_1=u_1,
\]
and
\begin{equation}\label{eq:X1def2}
  X_1
  =\frac{\gamma(t)}2
   +\frac{1}{3}u_1^{-2}D(u_1).
\end{equation}
Let
\[
  f_1:\R\langle x,y\rangle\longrightarrow A_1'
\]
be the homomorphism given by $f_1(x)=X_1$, $f_1(y)=Y_1$.

Since $h_1$ divides $H$, equation~\eqref{eq:Hdef2} gives in $B_1$
\[
  \left(\frac{D(u_1)}{3u_1^2}\right)^2
  =\frac{\gamma(t)^2}{4}+\frac{t^3-u_1^3}{3}.
\]
Using \eqref{eq:X1def2}, this is exactly
\[
  X_1^2-\gamma(t)X_1+\frac{u_1^3-t^3}{3}=0,
\]
so $f_1(\Phi_1)=0$.

Since $B_1$ is commutative,
\[
  Y_1^2X_1+Y_1X_1Y_1+X_1Y_1^2=3Y_1^2X_1.
\]
By \eqref{eq:X1def2},
\[
  3Y_1^2X_1-\frac32\gamma(t)Y_1^2=D(Y_1).
\]
Using \eqref{eq:Ddef2}, we obtain $f_1(\Phi_2)=0$.

Let $L_1=\im(f_1)$. Since $Y_1=u_1$, the generic fibre of $L_1$ is all of
$B_1$, so $L_1$ has rank $7$ over $\R$. It is finite over $\R$: the element
$Y_1$ is integral by $h_1(Y_1)=0$, and $X_1$ is integral over $\R[Y_1]$ by
\eqref{eq:Phi1drugi}. Thus $L_1$ is a finite torsion-free, hence free, $\R$-module of
rank $7$.

Modulo $t$, equation~\eqref{eq:h1modt2} gives $y^7=0$, while
\eqref{eq:Phi1drugi} gives
\[
  x^2=-\frac13y^3.
\]
Because $L_1/tL_1$ is commutative, the ideal generated by $X_{1}+tL_{1}$ and $Y_{1}+tL_{1}$ in $L_{1}/tL_{1}$ is
nilpotent. Hence $L_1/tL_1$ is local.

\subsection{ Conclusion.}\label{56}
 Let notation be as in Theorem \ref{New}, and put $f_{1}, f_{2}, f_{3}$, $A_{1}', A_{2}', A_{3}'$ be as above in this section.
  Put $f=f_{1}\oplus f_{2}\oplus f_{3}$,  so $f(x)=(f_{1}(x), f_{2}(x), f_{3}(x))$ and $f(y)=(f_{1}(y), f_{2}(y), f_{3}(y))$.
 
  We will first verify that the assumptions of Theorem \ref{New} hold, so there is a flat deformation from the algebra $\im (f)/t\im (f)$ to $A={\mathbb C}^{\oplus 7}\oplus M_{2}({\mathbb C})^{\oplus n-1}\oplus M_{3}({\mathbb C})$.

 We need to check the following: 

(1) $\dim _{\K}\im (f_{1})=7$,  $\dim _{\K}\im (f_{2})=4(n-1)$, $\dim_{\K}\im (f_{3})=9$. Observe that $\dim_{\K}\im (f_{3})=9$ follows from Lemma \ref{lem:f3rank2} and $\dim _{\K}\im (f_{1})=7$ follows since $h_{1}(z)$ has degree $7$. Observe also that  $\dim _{\K}\im (f_{2})=4(n-1)$ follows from Lemma \ref{lem:f2ranknew}, and since 
 $B_2=\K [z]/\langle q_{2}(z)\rangle$ because $q_{2}(z)$ has degree $n-1$.

(2) the boundedness condition $\operatorname{Im}(f)\subseteq t^{-\delta '}M\{t\}$ holds for some $\delta '$.
 This holds since  $L_1=\im (f_{1})$ is a finite $R$-module, hence is contained in $t^{-\delta ''}$ times a fixed lattice in $A_1'$ for some $\delta ''$;  $L_2=\im (f_{2})\subseteq M_2(R[u_2])$ once $c,d\in R[u_2]$ are established; $L_3=\im (f_{3})\subseteq M_3(R)$.

Then we can choose one $\delta '$ working for all three components.

 It remains to show that   $N$ (the completed Jacobi algebra) is isomorphic to $L/tL$. Observe first that these algebras have the same dimension in the case of contraction algebras. Indeed, 
\[
 \dim_{\C}N=\dim_{\C}(L/tL)=7+4(n-1)+9=4n+12.
\]
 It is known from  \cite{MW} that 
\[
 \dim_{\C}N=4n+12
\]
for the relevant type-$E$ contraction algebras, together with nilpotence of the ideal generated by $x,y$.  

 Observe also that by the construction, $  \Phi1, \Phi2  \in \ker (f)$. Therefore, 
 $L/tL$ satisfies the defining relations of the Jacobi algebra $3x^{2}+y^{3}$ and $xy^{2}+yxy+y^{2}x+g(y)$ which come from the cyclic derivatives of the superpotential.  Let $r \in x{\mathbb C}[[t]]\langle x, y\rangle+y{\mathbb C}[[t]]\langle x, y\rangle$. We will show that $f(r^{k})\in t\im (f)$, for some $k$. As shown in the previous subsections,  for $i=1,2,3$ there are $r_{i}\in {\mathbb C}[[t]]\langle x, y\rangle$ such that
 $f_{i}(r^{k}-tr_{i})=0$ for some $k$. Therefore $f((r^{k}-tr_{1})(r^{k}-tr_{2})(r^{k}-tr_{3}))=0$ for some $k$. Therefore, $f(r^{3k})\in t\im (f)$. Therefore, the ideal generated by images of $x$ and $y$ in $L/tL$ is nil. 
 Since $L/tL$ is finite-dimensional, it is nilpotent. 

 Therefore, $L/tL$ satisfies all the defining relations of $N$ (since it is the completed Jacobi algebra). Indeed, reasoning similarly as in Lemma \ref{nieuzyte} we obtain $I'=I+\mathfrak m^{4n+12}$, and 
\[
 N\cong\C\langle x,y\rangle/(I+\mathfrak m^{4n+12}).
\]
Since these two algebras have the same dimension it follows that $ N$ is isomorphic to $L/tL=\im (f)/t\im (f)$.

It follows that, for every polynomial $p(y)\in {\mathbb C}[y]$ with $\ord(p)=n$ with $n> 4$ and $n-1$ not divisible by $3$, the contraction algebra  given by the superpotential  \[x^{3}+xy^{3}+p(y),\] deforms flatly to the algebra
\[  {\mathbb C}^{\oplus 7}\oplus M_{2}({\mathbb C})^{\oplus n-1}\oplus M_{3}({\mathbb C}),\]
 provided that the completed Jacobi algebra has dimension $4n+12$ and that the ideal generated by $x,y$ is nilpotent.

\section{The case when $n-1>3$ is divisible by $3$}

 Let $\ord(p)=n$ and let $3$ divide $n-1$ in the ring of integers, so $n-1=3k$ for some natural number $k$. Since the 
present subsection treats the case $n-1>3$, we have $k>1$. Write $g$ according to congruence classes of exponents modulo $3k$:
\begin{equation}\label{eq:gdecomp}
  g(y)=g_0(y^{3k})+g_1(y^{3k})y+\cdots +g_{3k-1}(y^{3k})y^{3k-1},
\end{equation}
where $g_0,g_1, \cdots,  g_{3k-1}\in\C[z]$.
Since the order of $g(y)$ is $n-1=3k$, $g(y)\in y^{3k}{\mathbb C}[y]$. Hence, 
\[
  g_0(0)=g_1(0)=\cdots =g_{3k-1}(0)=0,
\]
and, because the coefficient of $y^{3k}$ in $g$ is $a_0\neq0$,
\[
  g_0(z)=a_0z+O(z^2).
\]

Choose once and for all
\[
  \sigma\in\C\setminus\{0,1\}
\]
and a primitive cube root of unity $ \omega\in\C$, so
\[
  \omega^3=1,\qquad \omega\neq1,
  \qquad 1+\omega+\omega^2=0.
\]
Define
\begin{align}
  s(t,y)&=\sum_{i=0}^{k-1}-g_{3i}(\sigma t^{3k})y^{3i},\\
  r(t,y)&=\sum_{i=0}^{k-1}-g_{3i+1}(\sigma t^{3k})y^{3i},\\
  \ell(t,y)&=\sum_{i=0}^{k-1}-g_{3i+2}(\sigma t^{3k})y^{3i},\\
  \gamma(t,y)&=\sum_{i=0}^{k-1}2\bigl(g_{3i+2}(t^{3k})-g_{3i+2}(\sigma t^{3k})\bigr)y^{3i},\label{eq:gamma}\\
  \alpha(t,y)&=t^{-3}\sum_{i=0}^{k-1}\bigl(g_{3i}(\sigma t^{3k})-g_{3i}(t^{3k})\bigr)y^{3i},\label{eq:alpha}\\
  \beta(t,y)&=t^{-3}\sum_{i=0}^{k-1}\bigl(g_{3i+1}(\sigma t^{3k})-g_{3i+1}(t^{3k})\bigr) y^{3i}.\label{eq:beta}
\end{align}
Then
\[
  s,r,\ell,\gamma\in t^{3k}\R[y],
  \qquad \alpha,\beta\in\R[y].
\]

We shall arrange that $f$ annihilates 
\begin{equation}\label{eq:Phi1newnew}
  \Phi_1
  =x^2-\gamma(t, y)x+\frac{1}{3}(y^3-t^3)
\end{equation}
and
\begin{equation}\label{eq:Phi2New}
\begin{split}
  \Phi_2={}&y^2x+yxy+xy^2+g(y)
  +s(t, y)+r(t, y)y+\ell(t, y)y^2
  -\frac{3}{2}\gamma(t, y)y^2.
\end{split}
\end{equation}
Modulo $t$, these are exactly the two Jacobi relations.

\subsection{The $3\times3$ component}

Put
\[
  A_3'=\Mat_3(\K).
\]
Denote \[\alpha (t)=\alpha (t,t), \beta (t)=\beta (t,t), \gamma (t)=\gamma (t,t).\]
Define
\[
  X_3=
  \begin{pmatrix}
    \gamma(t)&\alpha(t)&\beta(t)\\
    0&0&0\\
    0&0&0
  \end{pmatrix},
  \qquad
  Y_3=
  \begin{pmatrix}
    0&1&0\\
    0&0&1\\
    t^3 &0&0
  \end{pmatrix}.
\] 
Let
\[
  f_3:\R\langle x,y\rangle\longrightarrow A_3'
\]
be the unital homomorphism defined by $f_3(x)=X_3$ and $f_3(y)=Y_3$.

We have
\[
  Y_3^3=t^3I_3,
  \qquad
  X_3^2=\gamma(t)X_3.
\]
Hence $f_3(\Phi_1)=0$.

A direct multiplication gives
\begin{equation}\label{eq:f3mixed}
  Y_3^2X_3+X_3Y_3^2+Y_3X_3Y_3
  =t^3\alpha(t)I_3+t^3\beta(t)Y_3+\gamma(t)Y_3^2.
\end{equation}
Since $Y_3^3=t^3I_3$, equation~\eqref{eq:gdecomp} gives
\[
  g(Y_3)=\sum_{i=0}^{k-1}(g_{3i}(t^{3k})Y_{3}^{3i}I_3+g_{3i+1}(t^{3k})Y_3^{3i+1}+g_{3i+2}(t^{3k})Y_3^{3i+2}).
\]
 Recall that $Y_{3}^{3}=t^{3}I_{3}$.  
Using the definitions of $s,r,\ell,\alpha,\beta,\gamma$, the coefficients of
$I_3,Y_3,Y_3^2$ in $f_3(\Phi_2)$ all vanish. Thus
\[
  f_3(\Phi_2)=0.
\]
Let $L_3=\im(f_3)$.

\begin{lemma}\label{lem:f3rank3}
The $\K$-algebra generated by $X_3$ and $Y_3$ is $\Mat_3(\K)$. Moreover, the 
ideal of $L_{3}/tL_{3}$ generated by the images of $X_{3}$ and $Y_{3}$ is nilpotent. Hence $L_{3}/tL_{3}$ is
local.

\end{lemma}

\begin{proof}
 This is done similarly to the proof of Lemma \ref{lem:f3rank} and  in the few lines below this proof. In fact the matrices used in  Lemma \ref{lem:f3rank} have the same structure, but they have different $\alpha (t), \beta (t), \gamma (t)$. The results follow with the new $\alpha (t), \beta (t), \gamma (t)$ in the same way.
 Here is a direct argument that the determinant used in the generation argument  is nonzero: 
 Since $g_0(z)=a_{0}z+O(z^{2})$, one gets $\alpha (t)=a_{0}(\sigma -1)t^{3k-3}+O(t^{3k}),$
 whereas $\beta (t)=O(t^{3k-3})$, $\gamma (t)=O(t^{3k})$. 
 Therefore, in 
\[\gamma ^{3}+t^{3}\alpha ^{3}+t^{6}\beta ^{3}-3t^{3}\alpha \beta \gamma \]
the term $t^{3}\alpha ^{3}$ has order $9k-6$, while all other terms have order at least $9k-3$. 

More precisely, its leading term is $a_0^3(\sigma-1)^3t^{9k-6}$, so the expression is nonzero. 

 To show that the ideal generated by $X_{3}+tL_{3}$ and $Y_{3}+tL_{3}$ in $L_{3}/tL_{3}$ is nilpotent observe that, by using the relations from   Lemma \ref{lem:f3rank} one has $Y_3^3=0$, $X_3^2=0$, and because $\ord (g)=3k\ge6$, also $g(Y_3)=0$. Since $\gamma (t)$ reduces to zero modulo $t$, we also have the following relations modulo $t$:
\[X_{3}Y_{3}X_{3}=X_{3}Y_{3}^2X_{3}=0,
  \qquad
  Y_{3}^2X_{3}+Y_{3}X_{3}Y_{3}+X_{3}Y_{3}^2=0.
\]
A direct reduction gives the spanning set
\[
  1,\ X_{3},\ Y_{3},\ X_{3}Y_{3},\ Y_{3}X_{3},\ Y_{3}^2,\ Y_{3}X_{3}Y_{3},\ Y_{3}^2X_{3},\ Y_{3}^2X_{3}Y_{3}
\]
for $L_3/tL_3$. Since the reductions preserve the length of each word, it follows that all words of length larger than $5$ vanish. 

\end{proof}
\subsection{The $2\times2$ component}
 Recall that $k$ is an integer such that $n-1=3k$. 
Let
\[
  q_2(z)=z^{3k}-\sigma t^{3k},
  \qquad
  B_2=\K[z]/(q_2(z)),
  \qquad
  u_2=z+(q_2(z)),
\]
and put

\[
  A_2'=\Mat_2(B_2).
\]
Thus $u_2^{3k}=\sigma t^{3k}$ where $n-1=3k$.

Set
\begin{equation}\label{eq:cdef}
  c(t, u_{2})=\frac{\gamma(t, u_{2})^2}{4}+{\frac 13} (t^{3}-{u_{2}}^{3}).
\end{equation}
Let
\[
  \Lambda=\bigoplus_{i=0}^{3k-1}\R u_2^i.
\]
Since $q_2$ is monic, every element of $\R[u_2]$ has a unique representative in $\Lambda$. Since
$\gamma(t,u_2)\in t^{3k}\Lambda$, we have
\[
  \gamma(t,u_2)^2\in t^{6k}\Lambda.
\]
Because $k>1$, we have $3<3k$, so $u_2^3$ is one of the basis monomials in $\Lambda$. Hence
\[
  \frac13(t^3-u_2^3)\notin t^{6k}\Lambda.
\]
Consequently, 
\[
  c(t, u_{2})\neq0.
\]
Define
\[
  X_2=
  \begin{pmatrix}
    {\frac {\gamma(t, u_{2})}2}&1\\
    c(t, u_{2})&{\frac {\gamma(t, u_{2})}2}
  \end{pmatrix},
  \qquad
  Y_2=
  \begin{pmatrix}
    u_2&0\\
    0&\omega u_2
  \end{pmatrix}.
\]
 Recall that $\omega \neq 1$, ${\omega }^{3}=1$.
Let
\[
  f_2:\R\langle x,y\rangle\longrightarrow A_2'
\]
be defined by $f_2(x)=X_2$, $f_2(y)=Y_2$.

Put
\[
  Q=X_2-\frac{\gamma(t, u_{2})}2I_2
  =\begin{pmatrix}0&1\\c(t, u_{2})&0\end{pmatrix}.
\]
Then
\[
  Q^2=c(t, u_{2})I_2,
  \qquad
  Y_2^3=u_{2}^{3}I_2.
\]
By the definition of $c(t, u_{2})$,
\[
  Q^2
  =\left(\frac{\gamma(t, u_{2})^2}{4}+{\frac{1}{3}}(t^3-u_{2}^{3})\right)I_2,
\]
which is equivalent to $f_2(\Phi_1)=0$.

For the second relation, the off-diagonal entries of
\[
  Y_2^2X_2+Y_2X_2Y_2+X_2Y_2^2
\]
are multiples of
\[
  1+\omega+\omega^2=0,
\]
while its diagonal part is ${\frac 32}\gamma(t, u_{2})Y_2^2$. Therefore
\[
  Y_2^2X_2+Y_2X_2Y_2+X_2Y_2^2
  -{\frac32}\gamma(t, u_{2})Y_2^2=0.
\]
Also, because $Y_2^{3k}=\sigma t^{3k}I_2$,
 
\[
  g(Y_2)
  =(\sum_{i=0}^{k-1}g_{3i}(\sigma t^{3k})u_2^{3i})I_2
   +(\sum_{i=0}^{k-1}g_{3i+1}(\sigma t^{3k})u_2^{3i})Y_2
   +(\sum_{i=0}^{k-1}g_{3i+2}(\sigma t^{3k})u_2^{3i})Y_2^2.
\]
Hence $f_2(\Phi_2)=0$.

\begin{lemma}\label{lem:f2rank2}
The $\K$-algebra generated by $X_2$ and $Y_2$ is $A_2'=\Mat_2(B_2)$.
\end{lemma}

\begin{proof} 
{\em Part $1$.} We first show that $c(t, u_{2})$ is not a zero divisor in $B_2$.
Suppose to the contrary that $c(t, u_{2})$ is a zero divisor in $B_2$.
Since
\[
 q_2(z)=z^{3k}-\sigma t^{3k}
\]
has distinct roots over $\K$, there is a root $u'=t\lambda$, where $\lambda\in\C$ and $\lambda^{3k}=\sigma$, such that
\[
 c'=\frac{\gamma(t,u')^2}{4}+\frac13(t^3-{u'}^3)=0.
\]
Thus
\[
 c'=\frac{(1-\lambda^3)t^3}{3}+\frac{\gamma(t,t\lambda)^2}{4}.
\]
Observe that $\gamma(t,t\lambda)\in t^{3k}\R$, hence
$\gamma(t,t\lambda)^2\in t^{6k}\R$. Moreover, $\lambda^3\neq1$, since otherwise
\[
 \sigma=\lambda^{3k}=(\lambda^3)^k=1,
\]
contrary to the choice $\sigma\neq1$. Therefore $c'\neq0$, a contradiction. Hence $c(t, u_{2})$ is not a zero divisor in $B_2$.

{\em Part $2$.} Let $\mathcal A$ denote the $\K$-algebra generated by $X_2$ and $Y_2$.
Since $\gamma(t,y)$ contains only powers $y^{3i}$ and $Y_2^{3i}=u_2^{3i}I_2$, we have
\[
 \gamma(t,Y_2)=\gamma(t,u_2)I_2.
\]
Hence
\[
 Q=X_2-\frac12\gamma(t,Y_2)\in\mathcal A.
\]
Set
\[
  Z=QY_2Q
   =c(t, u_{2})\begin{pmatrix}\omega u_2&0\\0&u_2\end{pmatrix}.
\]
Also,
\[
 Q^2Y_2=c(t, u_{2})\begin{pmatrix}u_2&0\\0&\omega u_2\end{pmatrix}.
\]
These two diagonal matrices have linearly independent diagonal patterns because $\omega\neq1$. Hence their $\K$-linear combinations give
\[
 c(t, u_{2})\begin{pmatrix}u_2&0\\0&0\end{pmatrix},
 \qquad
 c(t, u_{2})\begin{pmatrix}0&0\\0&u_2\end{pmatrix}.
\]
By multiplying these matrices by powers of $Y_2$ and using
\[
 u_2^{3k}=\sigma t^{3k}\in\K^\times,
\]
we obtain, for $0\leq i\leq3k-1$,
\[
 c(t, u_{2})u_2^iE_{11},\qquad c(t, u_{2})u_2^iE_{22}\in\mathcal A.
\]
Right multiplication by $Q$ gives
\[
 c(t, u_{2})u_2^iE_{12},\qquad c(t, u_{2})^2u_2^iE_{21}\in\mathcal A.
\]
Since $1,u_2,\ldots,u_2^{3k-1}$ is a $\K$-basis of $B_2$, and multiplication by $c(t, u_{2})$ and $c(t, u_{2})^2$ is injective, these $12k$ matrices are $\K$-linearly independent. 

Since
\[
 \dim_{\K}\Mat_2(B_2)=12k,
\]
we conclude that $\mathcal A=\Mat_2(B_2)$.
\end{proof}

Let $L_2=\im(f_2)$. Modulo $tL_2$ the relations become 
\begin{equation}\label{eq:L2specialnew}
  3x^2+y^3=0,
  \qquad
  y^{n-1}=0,
  \qquad
  y^2x+yxy+xy^2=0.
\end{equation}

 Notice that $L_{2}/tL_{2}$ also satisfies these relations, so it is a homomorphic image of the algebra given by relations~\eqref{eq:L2specialnew}, call this algebra $C$. 
 We will now show that the algebra $C$  is local and its
ideal generated by $x$ and $y$ is nilpotent. The same holds for $L_{2}/tL_{2}$ since it is a homomorphic image of the algebra $C$.
 For
\[
 C=\C\langle x,y\rangle/(3x^2+y^3,\ y^{3k},\ y^2x+yxy+xy^2),
\]
use the reductions
\[
 x^2\to-\tfrac13y^3,
 \qquad
 xy^2\to-y^2x-yxy,
 \qquad
 xyxy\to yxyx,
 \qquad
 yxyxy\to y^2xyx,
 \qquad
 y^{3k}\to0.
\]
After resolving the standard overlaps, the normal monomials are of the form
\[
 y^a,\quad y^ax,\quad y^axy,\quad y^axyx,
 \qquad 0\le a<3k.
\]
All these reductions are non-decreasing in ordinary word length, while the normal words have bounded length.  Therefore every sufficiently long word is zero, proving that $(x,y)$ is nilpotent and $C$ is local.
\subsection{The commutative degree-seven component}
We will use a very similar proof to the previous one when calculating the commutative component. 
 We repeat the whole argument below in our new setting.
Define
\begin{equation}\label{eq:Ddef3}
  D(y)=\sum_{i=0}^{k-1}g_{3i}(\sigma t^{3k})y^{3i}+g_{3i+1}(\sigma t^{3k})y^{3i+1}+g_{3i+2}(\sigma t^{3k})y^{3i+2}-g(y).
\end{equation}
The second relation in the commutative component will give
\[
  x-\frac{\gamma(t, y)}2=\frac{D(y)}{3y^2}.
\]
Combining this with \eqref{eq:Phi1newnew} leads to the polynomial
\begin{equation}\label{eq:Hdef3}
  H(y)
  =D(y)^2
   -9y^4\left(
      \frac{\gamma(t, y)^2}{4}+\frac{t^3-y^3}{3}
    \right)
  \in\R[y].
\end{equation}

Modulo $t$ we have
\[
  D(y)\equiv-g(y)\pmod t,
  \qquad
  \gamma(t,y)\equiv0\pmod t,
\]
so
\begin{equation}\label{eq:Hmodt3}
  H(y)\equiv g(y)^2+3y^7
  =y^7\bigl(3+O(y)\bigr)
  \pmod t.
\end{equation}
The factor multiplying $y^7$ has nonzero constant term (which is $3$).
After multiplying $H$ by a nonzero constant if necessary, we may assume that $H$ is monic
 (since $D(y)^2$ supplies the leading term of $H$, whose leading coefficient is $e_m^2\in\C^{\times}$; indeed
 even though $\gamma(t, y)$ depends on $y$, $D(y)$ has degree $m\ge3k$ with leading coefficient $-e_m\in\C^{\times}$, whereas $\deg\bigl(y^4\gamma(t,y)^2\bigr)\le 6k-2<2m $
and the remaining $y^4(t^3-y^3)$ term has degree at most $7$).
By Hensel's lemma, there is a monic 
factor
\[
  h_1(y)\in\R[y],\qquad \deg h_1=7,
\]
such that
\begin{equation}\label{eq:h1modt3}
  h_1(y)\equiv y^7\pmod t
\end{equation}
and $h_1(y)$ divides $H(y)$ in $\R[y]$.
Moreover, if $H=h_1h_2$ is the corresponding Hensel factorization, then the complementary factor $h_2$ has a unit constant term by \eqref{eq:Hmodt3}. Since
\[
  H(0)=g_0(\sigma t^{3k})^2\neq0,
\]
we obtain $h_1(0)=H(0)/h_2(0)\neq0$.

Set
\[
  B_1=\K[z]/(h_1(z)),
  \qquad
  u_1=z+(h_1(z)),
  \qquad
  A_1'=B_1.
\]
Since $h_1(0)\neq0$, $u_1$ is invertible in $B_1$. Define
\[
  Y_1=u_1,
\]
and
\begin{equation}\label{eq:X1def3}
  X_1
  =\frac{\gamma(t, u_{1})}2
   +\frac{1}{3}u_1^{-2}D(u_1).
\end{equation}
Let
\[
  f_1:\R\langle x,y\rangle\longrightarrow A_1'
\]
be the homomorphism given by $f_1(x)=X_1$, $f_1(y)=Y_1$.

Since $h_1$ divides $H$, equation~\eqref{eq:Hdef3} gives in $B_1$
\[
  \left(\frac{D(u_1)}{3u_1^2}\right)^2
  =\frac{\gamma(t,u_1)^2}{4}+\frac{t^3-u_1^3}{3}.
\]
Using \eqref{eq:X1def3}, this is exactly
\[
  X_1^2-\gamma(t,u_1)X_1+\frac{u_1^3-t^3}{3}=0,
\]
so $f_1(\Phi_1)=0$.

Since $B_1$ is commutative,
\[
  Y_1^2X_1+Y_1X_1Y_1+X_1Y_1^2=3Y_1^2X_1.
\]
By \eqref{eq:X1def3},
\[
  3Y_1^2X_1-\frac32\gamma(t,Y_1)Y_1^2=D(Y_1).
\]
Using \eqref{eq:Ddef3}, we obtain $f_1(\Phi_2)=0$.

Let $L_1=\im(f_1)$. Since $Y_1=u_1$, the generic fibre of $L_1$ is all of
$B_1$, so $L_1$ has rank $7$ over $\R$. It is finite over $\R$: the element
$Y_1$ is integral by $h_1(Y_1)=0$, and $X_1$ is integral over $\R[Y_1]$ by
\eqref{eq:Phi1newnew}. Thus $L_1$ is a finite torsion-free, hence free, $\R$-module of
rank $7$.

Modulo $t$, equation~\eqref{eq:h1modt3} gives $y^7=0$, while
\eqref{eq:Phi1newnew} gives
\[
  x^2=-\frac13y^3.
\]
Because $L_1/tL_1$ is commutative, the ideal generated by $x$ and $y$ is
nilpotent. Hence $L_1/tL_1$ is local.

{\bf The conclusion.} The arguments corresponding to  {\em The full homomorphism and its rank} and {\em Identification of the special fibre} are analogous to those in the previous sections.

 By the previous subsections: (a) $\dim_ \K A'_1=7$, $\dim_ \K A'_2=4(n-1)$, $\dim_ \K A'_3=9$; (b) the three matrix sizes are $1,2,3$, so Theorem \ref{New} applies; (c) boundedness is proved as in Subsection \ref{56}; and (d) the nilpotence argument from Subsection \ref{56} applies as well.  Then, we check that the dimensions agree, $7+4(n-1)+9=4n+12$.

It follows that, for every polynomial $p(y)\in\mathbb C[y]$ with $\ord(p)=n>4$ and $3\mid(n-1)$, the contraction algebra  given by the superpotential  \[x^{3}+xy^{3}+p(y),\] deforms flatly to the algebra
\[  {\mathbb C}^{\oplus 7}\oplus M_{2}({\mathbb C})^{\oplus n-1}\oplus M_{3}({\mathbb C}).\]

\section{The main general statement}
  We are now ready to state our main result, which follows from the conclusions of the previous sections. 
\medskip

\begin{theorem}
Let $p(y)\in\C[y]$ satisfy $\ord (p)=n\ge4$, and let $N$ be the completed Jacobi algebra associated with
\[
 x^3+xy^3+p(y).
\]
Assume that $N$ is a contraction algebra of the type under consideration; in particular, assume that its maximal ideal generated by $x,y$ is nilpotent and that
\[
 \dim_{\C}N=\begin{cases}
 27, & n=4,\\
 4n+12, & n>4.
 \end{cases}
\]
Then
\[
 N \quad\text{admits a flat deformation to}\quad
 \begin{cases}
 \C^{\oplus6}\oplus M_2(\C)^{\oplus3}\oplus M_3(\C), & n=4,\\[2mm]
 \C^{\oplus7}\oplus M_2(\C)^{\oplus(n-1)}\oplus M_3(\C), & n>4.
 \end{cases}
\]
\medskip
\end{theorem}

 Methods analogous to those used in our paper might also be relevant in other situations where deformations of noncommutative algebras appear.  Interesting applications of deformation theory in noncommutative algebra, representation theory of algebras and Hopf algebras  were considered for example  in \cite{DSV, Sri, IG, IG2, HV, Shepler, SW}. Deformations of concrete algebras given generators and relations appear also in mathematical physics, for example in  \cite{jose1, jose2}. 

$ $

{\bf Acknowledgements.} The paper was checked by ChatGPT, and ChatGPT simplified some proofs. 
 We thank Michael Wemyss for suggesting open questions.

\end{document}